\documentclass[leqno]{amsart}
\usepackage{amsfonts,amssymb,amsmath,amsgen,amsthm}
\usepackage{hyperref,color}
\usepackage{pdfsync}
\usepackage{bbm}
\usepackage{bm}

\theoremstyle{plain}
\newtheorem{theorem}{Theorem}[section]
\newtheorem{definition}[theorem]{Definition}
\newtheorem{lemma}[theorem]{Lemma}
\newtheorem{corollary}[theorem]{Corollary}
\newtheorem{proposition}[theorem]{Proposition}
\newtheorem{hyp}[theorem]{Assumption}
\theoremstyle{remark}
\newtheorem{remark}[theorem]{Remark}

\newcommand{\ii}{\mathrm{i}}
\newcommand{\ud}{\mathrm{d}}

\def\R{{\mathbb R}}
\def\N{{\mathbb N}}

\def\({\left(}
\def\){\right)}
\def\<{\left\langle}
\def\>{\right\rangle}

\def\1{{\mathbf 1}}

\def\d{{\partial}}
\def\eps{\varepsilon}

\DeclareMathOperator{\diver}{div}

\numberwithin{equation}{section}

\date\today

\title[Global weak solutions for magnetic NLS]{Global, finite energy, weak solutions for the
 NLS with rough, time-dependent magnetic  potentials}

\author[P.~Antonelli]{Paolo Antonelli}
\address[P.~Antonelli]{Gran Sasso Science Institute -- GSSI \\ via Crispi 7 \\ 67100 L'Aquila (Italy).}
\email{paolo.antonelli@gssi.it}

\author[A.~Michelangeli]{Alessandro Michelangeli}
\address[A.~Michelangeli]{International School for Advanced Studies -- SISSA \\ via Bonomea 265 \\ 34136 Trieste (Italy).}
\email{alemiche@sissa.it}
\author[R.~Scandone]{Raffaele Scandone}
\address[R.~Scandone]{International School for Advanced Studies -- SISSA \\ via Bonomea 265 \\ 34136 Trieste (Italy).}
\email{rscandone@sissa.it}

\thanks{Partially supported by the 2014-2017 MIUR-FIR 
grant ``\emph{Cond-Math: Condensed Matter and Mathematical Physics}'' code RBFR13WAET}

\begin{document}
\begin{abstract}
We prove the existence of weak solutions in the space of energy for a class of non-linear Schr\"odinger equations in the presence of a external, rough, time-dependent magnetic potential. Under our assumptions it is not possible to study the problem by means of usual arguments like resolvent techniques or Fourier integral operators, for example. We use a parabolic regularisation and we solve the approximating Cauchy problem. This is achieved by obtaining suitable smoothing estimates for the dissipative evolution. The total mass and energy bounds allow to extend the solution globally in time. We then infer sufficient compactness properties in order to produce a global-in-time finite energy weak solution to our original problem. 
\\


\end{abstract}

\maketitle


\section{Introduction and main result}

In this work we study the initial value problem associated with the non-linear Schr\"odinger equation with magnetic potential
\begin{equation}\label{eq:magneticNLS}
i\d_tu=-(\nabla-\ii\,A)^2u+\mathcal N(u)
\end{equation}
in the unknown $u\equiv u(t,x)$, $t\in\R$, $x\in\R^3$, where 
\begin{equation}\label{hart_pp_nl}
\mathcal N(u)=\lambda_1|u|^{\gamma-1}u+\lambda_2(|\cdot|^{-\alpha}\ast|u|^2)u,\qquad
\begin{array}{l}
\gamma\in(1,5]\,, \\
\alpha\in (0,3)\,, \\
\lambda_1,\lambda_2\geqslant 0
\end{array}
\end{equation}
is a defocusing non-linearity, both of local (pure power) and non-local (Hartree) type, and
$A:\R\times\R^3\rightarrow\R^3$ is the external time-dependent magnetic potential. (The cases $\alpha=0$ and $\gamma=1$ would make $\mathcal{N}(u)$ a trivial linear term.)

The novelty here will be the choice of $A$ within a considerably larger class of rough potentials than what customarily considered in the literature so far -- as a consequence, we will be in the condition to prove the existence of global-in-time weak solutions, without attacking for the moment the general issue of the global well-posedness.

Concerning the non-linearity, in the regimes $\gamma\in(1,5)$ and $\alpha\in(0,3)$ we say that $\mathcal{N}(u)$ it is \emph{energy sub-critical}, while for $\gamma=5$ is \emph{energy critical}.
Given the defocusing character of the equation, it will not be restrictive henceforth to set $\lambda_1=\lambda_2=1$, and in fact all our discussion applies also to the case when one of such couplings is set to zero.

The relevance of equation \eqref{eq:magneticNLS} is hard to underestimate, both for the interest it deserves per se, given the variety of techniques that have been developed for its study, and for the applications in various contexts in physics. Among the latter, \eqref{eq:magneticNLS} is the typical effective evolution equation for the quantum dynamics of an interacting Bose gas  subject to an external magnetic field, and as such it can be derived in suitable scaling limits of infinitely many particles \cite{am_GPlim,S-2008,Benedikter-Porta-Schlein-2015,AO-GP_magnetic_lapl-2016volume}: in this context the $|u|^{\gamma-1}u$ term with $\gamma=3$ (resp., $\gamma=5$) arises as the self-interaction term due to a two-body (resp., three-body) inter-particle interaction of short scale, whereas the $(|\cdot|^{-\alpha}\ast|u|^2)u$ term accounts for a two-body interaction of mean-field type, whence its non-local character. 
On the other hand \eqref{eq:magneticNLS} arises also as an effective equation for the dynamics of quantum plasmas. Indeed, for densely charged plasmas, the pressure term in the degenerate (i.e. zero-temperature) electron gas is effectively given by a non-linear function of the electron charge density \cite{Haas-plasmas-2011}, which in the wave-function dynamics corresponds to a power-type non-linearity (see for instance \cite{ADM-2017} for more details).

%

In the absence of an external field ($A\equiv 0$), equation \eqref{eq:magneticNLS} has been studied extensively, and global well-posedness and scattering are well understood, both in the critical and in the sub-critical case (\cite{cazenave,CKSTT-2008-energycrit,Miao-Hartree-2007,Dodson-2013,Ginibre-Velo-1984TS,Ginibre-Velo-1998Prov,Fang-Han-Zheng-NLS-2011}). Such results are mainly based upon (variants of the) perturbation theory with respect to the linear dynamics, built on Strichartz estimates for the free Schr\"{o}dinger propagator (\cite{Ginibre-Velo-CMP1992,Keel-Tao-endpoint-1998}). However when $A\equiv\!\!\!\!\!/\; 0$ the picture is much less developed. 

The main mathematical difficulty is to obtain suitable dispersive and smoothing estimates for the \emph{linear} magnetic evolution operator, in order to exploit a standard fixed point argument where the non-linearity is treated as a perturbation.

For \emph{smooth} magnetic potentials, local-in-time Strichartz estimates were established under suitable growth assumptions \cite{Yajima-magStri-91,Mizutani-2014-superquadratic}, based on the construction of the fundamental solution for the magnetic Schr\"{o}dinger flow by means of the method of parametrices and time slicing a la Fujiwara \cite{Fujiwara-JAM1979}, together with Kato's perturbation theory. If the potential has some Sobolev regularity and is sufficiently small, then  Strichartz-type estimates were obtained \cite{Stefanov-MagnStrich-2007} by studying the parametrix associated with the derivative Schr\"odinger equation $\d_tu-i\Delta u+A\cdot\nabla u=0$, exploiting the methods developed by Doi in \cite{Doi-1994,Doi-1996}.
Global well-posedness of \eqref{eq:magneticNLS} and stability results in the case of suitable smooth potentials are proved in \cite{DeBouard1991,naka-shimo-2005,Michel_remNLS_2008}.



As far as \emph{non-smooth} magnetic potentials are concerned, magnetic Strichartz estimates are still available with a number of restrictions. When $A$ is time-independent, global-in-time magnetic Strichartz estimates were established by various authors under suitable spectral assumptions (absence of zero-energy resonances) on the magnetic Laplacian $A$ \cite{EGSchlag-2008,EGSchlag-2009,DAncona-Fanelli-StriSmoothMagn-2008}, or alternatively under suitable smallness of the so called non-trapping component of the magnetic field \cite{DAFVV-2010}, up to the critical scaling $|A(x)|\sim |x|^{-1}$. Counterexamples at criticality are also known  \cite{Fanelli-Garcia-2011}. In the time-dependent case, magnetic Strichartz estimates are available only under suitable smallness condition of $A$ \cite{Georgiev-Stefanov-Tarulli-2007,Stefanov-MagnStrich-2007}.

Beyond the regime of Strichartz-controllable magnetic fields very few is known, despite the extreme topicality of the problem in applications with potentials $A$ that are rough, have strong singularities locally in space, and have a very mild decay at spatial infinity, virtually a $L^\infty$-behaviour. This generic case can be actually covered, and global well-posedness for \eqref{eq:magneticNLS} was indeed established \cite{M-2015-nonStrichartzHartree}, by means of energy methods, as an alternative to the lack of magnetic Strichartz estimates. However, such an approach is only applicable to non-local non-linearities with energy sub-critical potential (in the notation of \eqref{eq:magneticNLS}: $\lambda_1=0$ and $\alpha\leqslant 2$), for it crucially relies on the fact that the non-linearity is then locally Lipschitz in the energy space, power-type non-linearities being instead way less regular and hence escaping this method. The same feature indeed allows to extend globally in time the well-posedness for the Maxwell-Schr\"odinger system in higher regularity spaces \cite{Nakamura-Wada-MaxwSchr_CMP2007}.

In this work we are concerned precisely with the generic case where in \eqref{eq:magneticNLS}  neither are the external magnetic fields Strichartz-controllable, nor can the non-linearity be handled with energy methods.

The key idea is then to work out first the global well-posedness of an initial value problem in which an additional source of smoothing for the solution is introduced, as the one provided by the magnetic Laplacian is not sufficient. In a recent work by the first author and collaborators \cite{ADM-2017}, placed in the closely related setting of non-linear Maxwell-Schr\"{o}dinger systems, the regularisation was provided by Yosida's approximation of the identity. Here, instead, we introduce a parabolic regularisation, in the same spirit of \cite{GuoNakStr-1995} for the Maxwell-Schr\"{o}dinger system. The net result is the addition of a heat kernel effect in the linear propagator, whence the desired smoothing.

At the removal of the regularisation by a compactness argument, we obtain one -- not necessarily unique -- global-in-time, weak solution with finite energy, which is going to be our main result (Theorem \ref{th:main} below). 

To be concrete, let us first state the conditions on the magnetic potential. 
\begin{hyp}\label{hyp:assum1_alt}
The magnetic potential $A$ belongs to one of the two classes $\mathcal{A}_1$ or $\mathcal{A}_2$ defined by
\[
 \begin{split}
  \mathcal{A}_1\;&:=\;\widetilde{\mathcal{A}}_1\cap\mathcal{R} \\
  \mathcal{A}_2\;&:=\;\widetilde{\mathcal{A}}_2\cap\mathcal{R} \,,
 \end{split}
\]
where
\[
 \widetilde{\mathcal{A}}_1\;:=\;
\left\{
A=A(t,x)\left|\!
\begin{array}{c}
 \mathrm{div}_x A=0\;\textrm{ for a.e.}\;t\in\mathbb{R}, \\
 A=A_1+A_2\textrm{ such that, for $j\in\{1,2\}$,} \\
 A_j\in L^{a_j}_\mathrm{loc}(\mathbb{R},L^{b_j}(\mathbb{R}^3,\mathbb{R}^3))  \\
 a_j\in(4, +\infty],\quad b_j\in(3, 6),\quad \frac{2}{\,a_j}+\frac{3}{\,b_j}<1
\end{array}\!\!
\right.\right\}
\]
and
\[
 \widetilde{\mathcal{A}}_2\;:=\;
\left\{
A=A(t,x)\left|\!
\begin{array}{c}
 \mathrm{div}_x A=0\;\textrm{ for a.e.}\;t\in\mathbb{R}, \\
 A=A_1+A_2\textrm{ such that, for $j\in\{1,2\}$,} \\
 A_j\in L^{a_j}_\mathrm{loc}(\mathbb{R},W^{1, \frac{3b_j}{3+b_j}}(\mathbb{R}^3,\mathbb{R}^3))  \\
 a_j\in(2, +\infty],\quad b_j\in(3, +\infty],\quad \frac{2}{\,a_j}+\frac{3}{\,b_j}<1
\end{array}\!\!
\right.\right\}\,,
\]
and where
\[
 \mathcal{R}\;:=\;\big\{ A\in\widetilde{\mathcal{A}}_1\textrm{ or }A\in\widetilde{\mathcal{A}}_2\:|\: \partial_t A_j\in L^1_{\mathrm{loc}}(\mathbb{R},L^{b_j}(\mathbb{R}^3,\mathbb{R}^3)),\,j=1, 2\big\}\,.
\]

Associated to such classes, we define
\begin{equation*}
\begin{aligned}
\|A\|_{\mathcal A_1}:=\;&\|A_1\|_{L^{a_1}_tL^{b_1}_x}+\|A_2\|_{L^{a_2}_tL^{b_2}_x}\\
\|A\|_{\mathcal A_2}:=\;&\|A_1\|_{L^{a_1}_tW^{1, \frac{3b_1}{3+b_1}}_x}+\|A_2\|_{L^{a_2}_tW^{1, \frac{3b_2}{3+b_2}}_x}\,.
\end{aligned}
\end{equation*}
\end{hyp}

%
%
%

A few observations are in order. First and foremost, both classes $\mathcal{A}_1$ and $\mathcal{A}_2$ include magnetic potentials for which in general the validity of Strichartz estimates for the magnetic Laplacian is not known. 

A large part of our intermediate results, including in particular the local theory in the energy space, are found with magnetic potentials in the larger classes $\widetilde{\mathcal{A}}_1$ and $\widetilde{\mathcal{A}}_2$. The mild amount of regularity in time provided by the intersection with the class $\mathcal{R}$  is needed to infer suitable a priori bounds on the solution from the estimates on the total energy. This allows one to extend globally in time the solution to the regularised problem. 

Regularity in time of the external potential is not needed either when equation \eqref{eq:magneticNLS} is studied in the \emph{mass sub-critical} regime, i.e., when $\gamma\in(1,\frac73)$ and $\alpha\in(0,2)$, and when $\max{\{b_1,b_2\}}\in (3,6)$. In this case we are able to work with the more general condition  $A\in\widetilde{\mathcal A}_1$. This is a customary fact in the context of Schr\"{o}dinger equations with time-dependent potentials, as well known since \cite{Yajima1987_existence_soll_SE} (compare Theorems \cite[Theorem 1.1]{Yajima1987_existence_soll_SE} and \cite[Theorem 1.4]{Yajima1987_existence_soll_SE} therein: $L^a$-integrability in time on the electric external potentials yields a $L^p$-theory in space, whereas additional $L^a$-integrability of the time derivative of the potential yields a $H^2$-theory in space). 
Our aim here of studying finite energy solutions to \eqref{eq:magneticNLS} thus requires some intermediate assumptions on the magnetic potential, determined by the class $\mathcal R$ above. See also Proposition 1.7 in \cite{Carles} where a similar issue is considered.

The additional requirement on $\nabla A$ present in the class $\mathcal{A}_2$ is taken to accommodate slower decay at infinity for $A$, way slower than the behaviour $|A(x)|\sim |x|^{-1}$ (and in fact even a $L^\infty$-behaviour) which, as mentioned before, is critical for the validity of magnetic Strichartz inequalities.

Last, it is worth remarking that the divergence-free condition, $\diver_x{A}=0$, is assumed merely for convenience: our entire analysis can be easily extended to the cases where $\diver_x{A}$ belongs to suitable Lebesgue spaces and consider it as a given (electrostatic) scalar potential.

Here is finally our main result. Clearly, there is no fundamental difference in studying solutions forward or backward in time, and as customary we shall only consider henceforth the problem for $t\geqslant 0$. Our entire discussion can be repeated for the case $t\leqslant 0$.

\begin{theorem}[Existence of global, finite energy weak solutions]\label{th:main}~

\noindent Let the magnetic potential $A$ be such that $A\in\mathcal A_1$ or $A\in\mathcal A_2$, and take $\gamma\in (1,5]$, $\alpha\in(0,3)$. Then, for every initial datum $f\in H^1(\R^3)$, the Cauchy problem
\begin{equation}\label{eq:CauMNLS}
\begin{split}
 & \begin{cases}
\;\;\ii\,\partial_t u\;=\; -(\nabla-\ii\,A)^2u+|u|^{\gamma-1}u+(|\cdot|^{-\alpha}*|u|^2)u \\
\:u(0,\cdot)=\;f
\end{cases} \\
& \quad t\in[0,+\infty),\;\; x\in\R^3
\end{split}
\end{equation}
admits a global weak $H^1$-solution
\begin{equation*}
u\;\in\; L_{\mathrm{loc}}^{\infty}([0,+\infty),H^1(\R^3))\cap W_{\mathrm{loc}}^{1,\infty}([0,+\infty),H^{-1}(\R^3))\,,
\end{equation*}
meaning that \eqref{eq:magneticNLS} is satisfied for a.e.~$t\in[0,+\infty)$ as an identity in $H^{-1}$ and $u(0,\cdot)=f$. Moreover, 
the energy
\begin{equation*}
 \begin{split}
  \mathcal{E}(u)(t)\;&:=\;\int_{\R^3}\Big({\textstyle{\frac{1}{2}}}|(\nabla-\ii A(t))\,u|^2+{\textstyle\frac{1}{\gamma +1}}|u|^{\gamma +1}+{\textstyle{\frac{1}{4}}}(|x|^{-\alpha}*|u|^2)|u|^2\Big)\,\ud x
 \end{split}
\end{equation*}
is finite and bounded on compact intervals.
\end{theorem}

In the remaining part of this Introduction, let us elaborate further on the general ideas behind our proof of Theorem \ref{th:main}.


As previously mentioned, we introduce a small dissipation term in the equation
\begin{equation}\label{eq:visc_nls}
 i\d_tu\;=\;-(1-\ii\,\eps)(\nabla-\ii\,A)^2u+\mathcal N(u)
\end{equation}
and we study the approximated problem.
Similar parabolic regularisation procedures are commonly used in PDEs, see for example the vanishing viscosity approximation in fluid dynamics or in systems of conservation laws, and in fact this was also exploited in a similar context by Guo-Nakamitsu-Strauss to study on the existence of finite energy weak solutions to the Maxwell-Schr\"odinger system \cite{GuoNakStr-1995}.

By exploiting the parabolic regularisation, we can now regard $i\d_tu+(1-\ii\,\eps)\Delta u$ as the main linear part in the equation and treat $(1-\ii\,\eps)(2\,\ii\,A\cdot\nabla u+|A|^2u)+\mathcal{N}(u)$ as a perturbation.

Evidently, this cannot be done in the purely Hamiltonian case $\eps=0$. Indeed, the term $A\cdot\nabla u$ is not a Kato perturbation of the free Laplacian and the whole derivative Schr\"odinger equation must be considered as the principal part \cite{Stefanov-MagnStrich-2007}. 

We can instead establish the local well-posedness in the energy space for the approximated Cauchy problem
\begin{equation}\label{eq:visc_CauMNLS}
\begin{split}
 & \begin{cases}
\:\ii\,\partial_t u\;=\; -(1-\ii\,\eps)(\nabla-\ii\,A)^2u+|u|^{\gamma-1}u+(|\cdot|^{-\alpha}*|u|^2)u \\
\:u(0,\cdot)\;=\;f 
\end{cases} \\
&\quad t\in[0,T]\,,\;\; x\in\R^3\,.
\end{split}
\end{equation}
We first obtain suitable Strichartz-type and smoothing estimates for the viscous magnetic evolution semi-group. This is done by exploiting the smoothing effect of the heat-Schr\"odinger semi-group $t\mapsto e^{(i+\varepsilon)t\Delta}$ and by inferring the same space-time bounds also for the viscous magnetic evolution, in a similar fashion as in \cite{Yajima-magStri-91,Naibo-Stefanov-MathAnn2006} scalar (electrostatic) potentials are treated as perturbations of the free Schr\"odinger evolution.

Next, the a priori bounds on the total mass and the total energy allow us to extend the solution of the regularised problem globally in time. It is worth stressing that such global well-posedness holds in the energy critical case too: indeed, when $\gamma=5$ the bounds deduced from the energy dissipation provide a uniform-in-time control on some Strichartz-type norms, and the argument is then completed by means of the blow-up alternative for the critical case.

The mass/energy a priori bounds turn out to be uniform in the regularising parameter $\eps>0$, which yields the needed compactness for the sequence of approximating solutions. It is then possible to remove the regularisation and to show the existence of a finite energy weak solution to our original problem \eqref{eq:CauMNLS}, at the obvious price of loosing the uniqueness, as well as its continuous dependence on the initial data.

The material is organized as follows: in Section \ref{sec:preliminaries} we collect the preliminary notions and results we need in our analysis. In particular, we clarify the notion of weak (and strong) $H^1$-solution and we derive suitable space-time estimates for the heat-Schr\"odinger evolution. In Section \ref{sec:propagator} we study the smoothing property of the magnetic linear Schr\"odinger equation with a parabolic regularisation. In Section \ref{sec:LWP} we prove local existence for the regularised magnetic non-linear Schr\"{o}dinger equation \eqref{eq:visc_nls}. In Section \ref{eq:mass-energy} we prove mass and energy estimates for \eqref{eq:visc_nls} together with certain a priori bounds. In Section \ref{eq:eps-GWP} we use the energy estimates and the a priori bounds to extend the solution (forward) globally in time, both in the energy sub-critical and critical case. In Section \ref{sec:remov-reg}, using a compactness argument, we remove the regularisation, eventually proving the main theorem.

\section{Preliminaries and notation}\label{sec:preliminaries}
In this Section we collect the definitions and main tools that we shall use in the rest of the work.

We begin with a few remarks on our notation. For two positive quantities $P$ and $Q$, we write $P\lesssim Q$ to mean that $P\leqslant CQ$ for some constant $C$ independent of the variables or of the parameters which $P$ and $Q$ depend on, unless explicitly declared; in the latter case we write, self-explanatorily, $P\lesssim_{\alpha} Q$, and the like. Given $p_1,\ldots p_n\in [1,+\infty]$, we define $p=p_1*p_2\ldots *p_n$ by
$$ \frac1p=\frac{1}{p_1}+\frac{1}{p_2}+\ldots +\frac{1}{p_n}\,.$$
The same operation can be extended component-wise to vectors in $[1,+\infty]^d$, and we still denote it by $*$. Thus, for example, $(s,p)=(s_1,p_1)*(s_2,p_2)$ will mean $s^{-1}=s_1^{-1}+s_2^{-1}$ and $p^{-1}=p_1^{-1}+p_2^{-1}$. Given $p\in[1,+\infty]$, we denote by $p'$ its H\"{o}lder dual exponent, defined by $p*p'=1$. Henceforth, we use the symbols $\diver$, $\nabla$ and $\Delta$ to denote derivations in the spatial variables only. When referring to the vector field $A:\mathbb{R}^3\to\mathbb{R}^3$, conditions like  $A\in L^p(\mathbb{R}^3)$ are to be understood as $A\in L^p(\mathbb{R}^3,\mathbb{R}^3)$. As customary, in a self-explanatory manner we  will frequently make only the dependence on $t$ explicit in symbols such as $A(t)$, $u(t)$, $\mathcal{N}(u(t))$, $(\nabla-\ii\,A(t))u$, etc., instead of writing $A(t,x)$, $u(t,x)$, $(\mathcal{N}(u))(t,x)$, $((\nabla-\ii\,A)u)(t,x)$, etc. The short-cut `NLS' refers as usual to non-linear Schr\"{o}dinger equation, in the sense that will be specified in the following. For sequences and convergence of sequences, we write $(u_n)_n$ and $u_n\to u$ for $(u_n)_{n\in\mathbb{N}}$ and $u_n\to u$ as $n\to+\infty$.

\subsection{Magnetic Laplacian and magnetic Sobolev space}

We clarify now the meaning of the symbol $(\nabla-\ii A(t))^2$. As mentioned already in the Introduction, formally 
$$(\nabla-\ii A(t))^2=\Delta-2\,\ii\,A(t)\cdot\nabla-i\diver{A(t)}-|A(t)|^2\,.$$ 
In our setting of divergence-free magnetic potentials, this becomes
$$(\nabla-\ii A(t))^2=\Delta-2\,\ii\,A(t)\cdot\nabla-|A(t)|^2.$$
If $A(t)\in L^2_{\mathrm{loc}}(\R^3)$ for almost every $t\in\R$, which will always be our case, then we define the \emph{magnetic Laplacian} $(\nabla-\ii A(t))^2$ as a (time-dependent) distributional operator, according to the following straightforward Lemma.
\begin{lemma}[Distributional meaning of the magnetic Laplacian]
Assume that, for almost every $t\in\R$, $A(t)\in L^2_{\mathrm{loc}}(\R^3)$ with $\diver{A}(t)=0$. Then for almost every $t\in\R$, $(\nabla-\ii A(t))^2$ is a map from $L^1_{\mathrm{loc}}(\R^3)$ to $\mathcal{D}'(\R^3)$, which acts on a generic $f\in L^1_{\mathrm{loc}}(\R^3)$ as
$$(\nabla-\ii A(t))^2f\;=\;\Delta f-2\,\ii\,A(t)\cdot\nabla f-|A(t)|^2f\,.$$
\end{lemma}

In order to qualify such a distribution as an element of a suitable functional space, it is natural to deal with the magnetic Sobolev space defined as follows. (Here, with respect to our general setting, $A$ is meant to be a magnetic vector potential at a fixed time.)

\begin{definition}\label{de:magn_space}
Let $A\in L_{\mathrm{loc}}^2(\R^3)$. We define \emph{magnetic Sobolev space}
$$H_A^{1}(\R^3):=\{f\in L^2(\R^3)\,|\,(\nabla-\ii\,A)f\in L^2(\R^d)\}$$ 
equipped with the norm
$$\Vert f\Vert_{H_A^{1}(\R^3)}^2\;:=\;\Vert f\Vert_{L^2(\R^d)}^2+\Vert(\nabla-\ii\,A)f\Vert_{L^2(\R^d)}^2\,,$$
which makes $H_A^{1}(\R^3)$ a Banach space.
\end{definition}

We recall \cite[Theorem 7.21]{Lieb-Loss-Analysis} that, when $A\in L_{\mathrm{loc}}^2(\R^3)$, any $f\in H^1_A(\R^3)$ satisfies the \emph{diamagnetic inequality}
\begin{equation}\label{eq:diam}
|(\nabla |f|)(x)|\leq|((\nabla-\ii\,A)f)(x)|\quad\mbox{for a.e. }x\in\R^3\,.
\end{equation}

The following two Lemmas express useful magnetic estimates in our regime for $A$.

\begin{lemma}
Assume that $A\in\mathcal{A}_1$ or $A\in\mathcal{A}_2$. Then, for almost every $t\in\R$, 
\begin{equation}\label{H1H-1}
\Vert 2\,\ii\, A(t)\cdot\nabla f+|A(t)|^2f\Vert_{H^{-1}(\R^3)}\;\lesssim\; C_A(t)\Vert f\Vert_{H^1(\R^3)}\,,
\end{equation}
where
\begin{equation*}
C_A(t)\;:=\;1+\Vert A_1(t)\Vert^2_{L^{b_1}(\R^3)}+\Vert A_2(t)\Vert^2_{L^{b_2}(\R^3)}\,.
\end{equation*}
In particular, for almost every $t\in\R$, $(\nabla-\ii A(t))^2$ is a continuous map from $H^1(\R^3)$ to $H^{-1}(\R^3)$. 
\end{lemma}

\begin{proof}
The proof is based on a straightforward application of Sobolev's embedding and H\"older's inequality.
\end{proof}

\begin{lemma}\label{le:norm_equiv}
Let $A\in L^b(\R^3)$ with $b\in[3,+\infty]$.
\begin{itemize}
\item[(i)] One has
\begin{equation}\label{HAem}
\|f\|_{L^q(\R^3)}\lesssim \|f\|_{H_A^{1}(\R^3)}\,,\quad q\in[2,6]
\end{equation}
with the constant in \eqref{HAem} independent of $A$, hence the embedding $H_A^{1}(\R^3)\hookrightarrow L^q(\R^3)$ for $q\in[2,6]$.
\item [(ii)] One has
\begin{equation}\label{eq:norm_equiv}
(1+\|A\|_{L^b(\R^3)})^{-1}\|f\|_{H^{1}(\R^3)}\;\lesssim\; \|f\|_{H_A^{1}(\R^3)}\;\lesssim\; (1+\|A\|_{L^b(\R^3)})\|f\|_{H^{1}(\R^3)}\,,
\end{equation}
whence $H_A^{1}(\R^3)\cong H^{1}(\R^3)$ as an isomorphism between Banach spaces
\end{itemize}
\end{lemma}

\begin{proof}
The proof is based on a straightforward application of Sobolev's embedding, H\"older's inequality, and the diamagnetic inequality.
\end{proof}

\begin{remark}
As an immediate consequence of Lemma \ref{le:norm_equiv}, given a potential $A\in\widetilde{\mathcal{A}}_1$ or $A\in \widetilde{\mathcal{A}}_2$, for almost every $t>0$ the magnetic Sobolev spaces $H_{A(t)}^1(\R^3)$ are all equivalent to the ordinary Sobolev space $H^1(\R^3)$.
\end{remark}


\subsection{Notion of solutions}

We give now the precise notion of strong and weak solutions for the Cauchy problem \eqref{eq:CauMNLS} and its regularised version \eqref{eq:visc_CauMNLS}.

For the sake of a comprehensive discussion, let us consider the general Cauchy problem
\begin{equation}\label{eq:ge_cau}
\begin{split}
&\begin{cases}
\;\;\ii\,\partial_t u\;=\;c\,(\Delta u-2\,\ii\,A(t)\cdot\nabla u-|A(t)|^2u)+\mathcal{N}(u) \\
u(0,\cdot)\:=\;f
\end{cases} \\
&\qquad t\in I:=[0,T)\,,\;\;  x\in\mathbb{R}^3\,,
\end{split}
\end{equation}
for some $T>0$ and  $c\in\mathbb C$ with $\mathfrak{Im}\,c\geqslant 0$. Here the choices $c=-1$ and $c=-1+i\eps$ correspond, respectively, to \eqref{eq:CauMNLS} and \eqref{eq:visc_CauMNLS}.

\begin{definition}\label{de:ws_solution_alt}
Let $I:=[0,T)$ for some $T>0$. Given an initial datum $f\in H^1(\R^3)$, we say that
\begin{itemize}
\item[(i)] a local strong $H^1$-solution $u$ to \eqref{eq:ge_cau} on $I$ is a function 
\[
 u\;\in\;\mathcal{C}(I,H^1(\R^3))\cap\mathcal C^1(I; H^{-1}(\R^3))
\]
such that $\ii\,\partial_t u=c\,(\Delta u-2\,\ii\,A(t)\cdot\nabla u-|A(t)|^2u)+\mathcal{N}(u)$ in $H^{-1}(\mathbb{R}^3)$ for all $t\in I$ and $u(0)=f$; 
\item[(ii)]  a local weak $H^1$-solution $u$ to \eqref{eq:ge_cau} on $I$ is a function 
\[
 u\;\in\;L^\infty(I,H^1(\R^3))\cap W^{1, \infty}(I; H^{-1}(\R^3))
\]
such that $\ii\,\partial_t u=c\,(\Delta u-2\,\ii\,A(t)\cdot\nabla u-|A(t)|^2u)+\mathcal{N}(u)$ in $H^{-1}(\mathbb{R}^3)$
for a.e.~$t\in I$ and $u(0)=f$.
\end{itemize}
Moreover, a function $u\in L^\infty_{\mathrm{loc}}([0,+\infty),H^1(\mathbb{R}^3))$ is called
\begin{itemize}
 \item[(iii)] a global   strong $H^1$-solution $u$ to \eqref{eq:ge_cau} if it is a local strong solution for every interval $I=[0,T)$;
 \item[(iv)] a global weak $H^1$-solution $u$ to \eqref{eq:ge_cau} if it is a local weak solution for every interval $I=[0,T)$.
\end{itemize}

\end{definition}

Next, we recall the notion of local and global well-posedness (\cite[Section 3.1]{cazenave}).

\begin{definition}\label{de:WP}
We say that equation
\[
 \ii\,\partial_t u\;=\;c\,(\Delta u-2\,\ii\,A(t)\cdot\nabla u-|A(t)|^2u)+\mathcal{N}(u)
\]
is locally well-posed in $H^1(\R^3)$ if the following conditions hold:
\begin{itemize}
\item[(i)] For any initial datum $f\in H^1(\R^3)$, the Cauchy problem \eqref{eq:ge_cau} admits a unique local strong $H^1$-solution, defined on a maximal interval $[0,T_{\mathrm{max}})$, with $T_{\mathrm{max}}=T_{\mathrm{max}}(f)\in(0,+\infty]$.
\item[(ii)] One has continuous dependence on the initial data, i.e., if $f_n\rightarrow f$ in $H^1(\R^3)$ and $0\ni I\subset[0,T_{\mathrm{max}})$ is a closed interval, then the maximal strong $H^1$-solution of \eqref{eq:ge_cau} with initial datum $f_n$  is defined on $I$ for $n$ large enough and satisfies $u_n\rightarrow u$ in $\mathcal{C}(I,H^1(\R^3))$.
\item[(iii)] In the energy-sub-critical case one has the blow-up alternative: if $T_{\mathrm{max}}<+\infty$, then  
$$\lim_{t\uparrow T_{\mathrm{max}}}\Vert u(t,\cdot)\Vert_{H^1(\R^3)}\;=\;+\infty\,.$$
\end{itemize}
We say that the same equation  is globally well-posed in $H^1(\R^3)$  if it is locally well-posed and if for any initial datum $f\in H^1(\R^3)$ the Cauchy problem \eqref{eq:ge_cau} admits a global strong $H^1$-solution.
\end{definition}

\subsection{Smoothing estimates for the heat-Schr\"odinger flow}\label{sec:smoothingestHSflow}

Let us now analyse the smoothing properties of the heat and the Schr\"odinger flows generated by the free Laplacian.

We begin by recalling the well-known dispersive estimates for the Schr\"odinger equation
\begin{equation}\label{disps}
\|e^{\ii t\Delta}f\|_{L^p(\R^3)}\;\lesssim\;|t|^{-\frac32\left(\frac{1}{p'}-\frac1p\right)}\|f\|_{L^{p'}(\R^3)}\,, \quad p\in[2,+\infty]\,,\;\;t\neq 0\,,
\end{equation}
and the $L^p-L^r$ estimates for the heat flow
\begin{eqnarray}
 \|e^{t\Delta}f\|_{L^r(\R^3)}\!\!\!&\lesssim& \!\!t^{-\frac{3}{2}\left(\frac{1}{p}-\frac1r\right)}\|f\|_{L^{p}(\R^3)}  \label{disph}\\
 & & \qquad\qquad\qquad\qquad\qquad 1\leqslant p\leqslant r\leqslant +\infty,\,t>0\,. \nonumber\\
 \|\nabla e^{t\Delta}f\|_{L^r(\R^3)}\!\!\!&\lesssim& \!\!t^{-\frac{3}{2}\left(\frac{1}{p}-\frac1r\right)-\frac12}\|f\|_{L^{p}(\R^3)} \label{disphg}
\end{eqnarray}

We also recall the definition of admissible pairs for the Schr\"odinger flow in three dimensions.

\begin{definition}\label{destro}
A pair $(q,r)$ is called admissible if
\begin{equation*}
\frac2q+\frac{3}{r}=\frac{3}{2},\qquad r\in[2,6]\,.
\end{equation*}
The pair $(2,6)$ is called endpoint, while the others are called non-endpoint. The pair $(s,p)$ is called dual-admissible if $(s,p)=(q',r')$ for some admissible pair $(q,r)$, namely
\begin{equation*}
\frac2s+\frac{3}{p}=\frac{7}{2},\qquad p\in[{\textstyle{\frac{6}{5}}},2]\,.
\end{equation*}
\end{definition}

The dispersive estimate \eqref{disps}  yields a whole class of space-time estimates for the Schr\"odinger flow \cite{Ginibre-Velo-1998Prov, Yajima1987_existence_soll_SE, Keel-Tao-endpoint-1998}.

\begin{proposition}[Strichartz estimates]\label{freestr}~
\begin{itemize}
\item[(i)] For any admissible pair $(q,r)$, the following homogeneous estimate holds:
\begin{equation}\label{eq:hofree}
\|e^{\ii t\Delta}f\|_{L^q(\R;L^r(\R^3))}\;\lesssim\;\|f\|_{L^2(\R^3)}\,.
\end{equation}
\item[(ii)] Let $I$ be an interval of $\R$ (bounded or not), and $\tau,t\in \overline{I}$. For any admissible pair $(q,r)$ and any dual admissible pair $(s,p)$, the following inhomogeneous estimate holds:
\begin{equation}\label{eq:rifree}
\Big\|\int_{\tau}^t e^{\ii(t-\sigma)\Delta}F(\sigma)\,\,\ud\sigma\Big\|_{L^q(I;L^r(\R^3))}\;\lesssim\;\|F\|_{L^s(I;L^p(\R^3))}\,.
\end{equation}
\end{itemize}
\end{proposition}

Similarly (see, e.g., \cite[Section 2.2.2]{Wang-Zhaohui-Chengchun_HarmonicAnalysisI_2011}), by means of  \eqref{disph}-\eqref{disphg} one infers an analogous class of space-time estimates for the heat propagator.

\begin{proposition}[Space-time estimates for $e^{t\Delta}$]~
\begin{itemize}
\item[(i)]  For any admissible pair $(q,r)$, the following homogeneous estimate holds:
\begin{equation}\label{eq:hofree_heat}
\|e^{t\Delta}f\|_{L^q([0,+\infty),L^r(\R^3))}\;\lesssim\;\|f\|_{L^2(\R^3)}\,.
\end{equation}
\item[(ii)] Let $I\subseteq \R$ be an interval of the form $[\tau,T)$, with $T\in(\tau,+\infty]$. For any admissible pair $(q,r)$ and any dual admissible pair $(s,p)$, the following inhomogeneous estimate holds:
\begin{equation}\label{eq:rifree_heat}
\Big\|\int_{\tau}^t e^{(t-\sigma)\Delta}F(\sigma)\,\,\ud\sigma\Big\|_{L^q(I;L^r(\R^3))}\;\lesssim\;\|F\|_{L^s(I;L^p(\R^3))}\,.
\end{equation}
\end{itemize}
\end{proposition}


We can also combine the previous results in order to infer $L^p$-$L^r$ estimates (Proposition \ref{prop:pw_ex_lemma}) and space-time estimates (Proposition \ref{le:stri_hs}) for the heat-Schr\"o\-dinger propagator. 

\begin{proposition}[Pointwise-in-time estimates for the heat-Schr\"odinger flow]\label{prop:pw_ex_lemma}~

\noindent For any $t>0$, $p\in[1,2]$, and $r\in[2,+\infty]$,
\begin{eqnarray}
 \|e^{(\ii+\eps)t\Delta}f\|_{L^r(\R^3)}\!\!&\lesssim& \!\!\eps^{-\frac{3}{2}|\frac{1}{p'}-\frac{1}{r}|}t^{-\frac{3}{2}(\frac1p-\frac1r)}\|f\|_{L^p(\R^3)} \label{hsdisp} \\
 \|\nabla e^{(\ii+\eps)t\Delta}f\|_{L^r(\R^3)}\!\!&\lesssim&\!\!  \eps^{-\frac{3}{2}|\frac{1}{p'}-\frac{1}{r}|-\frac12}t^{-\frac{3}{2}(\frac1p-\frac1r)-\frac12}\|f\|_{L^p(\R^3)}\,. \label{hsdispgr}
\end{eqnarray}
\end{proposition}

\begin{proof}
The proof is straightforward and follows by combining the decay estimates of both the heat and the Schr\"odinger propagators, see formulas \eqref{disps}-\eqref{disphg} above. In fact, similar decay estimates follow also by simply ignoring the hyperbolic part given by the Schr\"odinger evolution, however with a worse control in terms of $\eps$.
\end{proof}

\begin{proposition}[Space-time estimates for the heat-Schr\"odinger flow]\label{le:stri_hs}~

Let $\eps>0$ and let $(q,r)$ be an admissible pair. 
\begin{itemize}
 \item[(i)]  One has (homogeneous Strichartz estimate)
\begin{equation}\label{eq:homstrich}
\|e^{(\ii+\eps)t\Delta}f\|_{L^q([0,T],L^r(\R^3))}\;\lesssim\;\|f\|_{L^2(\R^3)}\,.
\end{equation}
\item[(ii)] Let $T>0$ and let the pair $(s,p)$ satisfy
\begin{equation}\label{est_dual}
\frac{2}{s}+\frac{3}{p}=\frac{7}{2},\qquad
\begin{cases}
\frac{1}{2}\leqslant\frac{1}{p}\leqslant 1&2\leqslant r<3\\
\frac{1}{2}\leqslant\frac{1}{p}<\frac{1}{r}+\frac{2}{3}&3\leqslant r\leqslant 6\,.
\end{cases}
\end{equation}
Then (inhomogeneous retarded Strichartz estimate)
\begin{equation}\label{eq:st}
\Big\|\int_{0}^{t} e^{(\ii+\eps)(t-\tau)\Delta}F(\tau)\,\,\ud\tau\Big
\|_{L^q([0,T],L^r(\R^3))}\;\lesssim_{\,\eps}\;\|F\|_{L^s([0,T],L^p(\R^3))}\,.
\end{equation}
\item[(iii)] Assume in addition that $(q,r)$ is non-endpoint. Let $T>0$ and let the pair $(s,p)$ satisfy
\begin{equation}\label{est_grad_dual}
\frac{2}{s}+\frac{3}{p}=\frac{5}{2},\qquad\frac{1}{2}\leqslant\frac{1}{p}<\frac{1}{r}+\frac{1}{3}\,.
\end{equation}
Then (inhomogeneous retarded Strichartz estimate)
\begin{equation}\label{eq:grst}
\Big\|\nabla\!\!\int_{0}^{t} e^{(\ii+\eps)(t-\tau)\Delta}F(\tau)\,\,\ud\tau\,\Big\|_{L^q([0,T],L^r(\R^3))}\;\lesssim_{\,\eps}\;\|F\|_{L^s([0,T],L^p(\R^3))}\,.
\end{equation}
\end{itemize}
\end{proposition}

\begin{remark}\label{re:comp_epsno} In \eqref{est_dual} the range of admissible pairs $(s,p)$ is \emph{larger} as compared to the case of the Schr\"odinger equation. 
In fact, dispersive equations, even if hyperbolic, have the remarkable property of enjoying a class of smoothing estimates.
More specifically, for the Schr\"odinger equation it can be proved that the inhomogeneous part in the Duhamel formula enjoys the gain of regularity by one derivative in space, see Theorem 4.4 in \cite{Lineares-Ponce-book2015}. However, it is straightforward to check that estimate \eqref{eq:grst} for the heat-Schr\"odinger semi-group is stronger than estimate (4.26) in \cite{Lineares-Ponce-book2015} and it is better suited to study our problem.
\end{remark}

\begin{proof}[Proof of Proposition \ref{le:stri_hs}] 
We begin with the proof of part (i). Combining the homogeneous Strichartz estimates \eqref{eq:hofree} for the Schr\"odinger flow with the estimate $\|e^{\eps t\Delta}f\|_{L^r(\R^3)}\lesssim\|f\|_{L^r(\R^3)}$, which follows by \eqref{disph}, we get
\begin{align*}
 \|e^{(\ii+\eps)t\Delta}f\|_{L^q([0,+\infty),L^r(\R^3))}\;&=\; \|e^{\eps t\Delta}e^{\ii t\Delta}f\|_{L^q([0,+\infty),L^r(\R^3))}\\
&\lesssim\; \|e^{\ii t\Delta}f\|_{L^q([0,+\infty),L^r(\R^3))}\;\lesssim\;\|f\|_{L^2(\R^3)}\,,
\end{align*}
which proves \eqref{eq:homstrich}. 
Next we prove part (ii). In the special case $(q,r)=(+\infty,2)$ and $(s,p)=(1,2)$, the dispersive estimates \eqref{hsdisp} yields
\begin{equation}\label{fi_pi}
 \begin{aligned}
&\Big\|\int_{0}^{+\infty} e^{(\ii+\eps)|t-\tau|\Delta}F(\tau)\,\,\ud\tau\,\Big\|_{L^{\infty}([0,T],L^2(\R^3))}\\
&\qquad\lesssim \int_{0}^{+\infty}\|F(\tau)\|\,\ud\tau\;=\;\|F\|_{L^{1}([0,T],L^2(\R^3))}\,.
\end{aligned}
\end{equation}
For the generic case, namely $(p,r)\neq(2,2)$, owing to \eqref{hsdisp} one obtains
\begin{align*}
&\Big\|\int_{0}^{+\infty} e^{(\ii+\eps)|t-\tau|\Delta}F(\tau)\,\,\ud\tau\,\Big\|_{L^q([0,T],L^r(\R^3))}\\
&\qquad\lesssim_{\,\eps}\;\Big\|\int_{0}^{+\infty}|t-\tau|^{-\gamma}\|F(\tau)\|_{L^p(\R^3)}\,\,\ud\tau\,\Big\|_{L^q[0,T]}\,,
\end{align*}
where $\gamma:=\frac{3}{2}(\frac1p-\frac1r)\in(0, 1)$ by the assumptions on $p,r$. The Hardy-Littlewood-Sobolev inequality in time yields then
\begin{equation}\label{se_pi}
\Big\|\int_{0}^{+\infty} e^{(\ii+\eps)|t-\tau|\Delta}F(\tau)\,\,\ud\tau\,\Big\|_{L^q([0,T],L^r(\R^3))}\;\lesssim_{\,\eps}\;\|F\|_{L^s([0,T],L^p(\R^3))}
\end{equation}
with $\frac1s=1+\frac1q-\gamma$, namely $\frac{2}{s}+\frac{3}{p}=\frac{7}{2}$.
Now, using estimates \eqref{fi_pi}-\eqref{se_pi} and the Christ-Kiselev lemma \cite{Christ-Kiselev-2001}, we deduce the ``retarded estimates'' \eqref{eq:st}.
The proof of part (iii) proceeds similarly as for part (ii). Indeed, owing to the dispersive estimate with gradient \eqref{hsdispgr} and the Hardy-Littlewood-Sobolev inequality in time,
\begin{align*}
&\Big\|\,\nabla\!\!\int_{0}^{+\infty} e^{(\ii+\eps)|t-\tau|\Delta}F(\tau)\,\,\ud\tau\,\Big\|_{L^q([0,T],L^r(\R^3))}\\
&\qquad\lesssim_{\,\eps}\;\Big\|\int_{0}^{+\infty}|t-\tau|^{-\gamma}\|F(\tau)\|_{L^p(\R^3)}\,\,\ud\tau\,\Big\|_{L^q(0,+\infty)} \;\lesssim_{\,\eps}\;\|F\|_{L^s([0,T],L^p(\R^3))}\,,
\end{align*}
where now $\gamma=\frac{3}{2}\left(\frac1p-\frac1r\right)+\frac12\in(0, 1)$ and the exponent $s$ is given by $\frac{2}{s}+\frac{3}{p}=\frac{5}{2}$; this, and again the result by Christ-Kieselev, then imply \eqref{eq:grst}.
\end{proof}


For our analysis it will be necessary to apply the above Strichartz estimates for the heat-Schr\"{o}dinger flow in a regime of indices that guarantees also to control the smallness of the constant in each such inequalities in terms of the smallness of $T$. This leads us to introduce the following admissibility condition.

\begin{definition}\label{de:ad-grad_ad}
 Let $(q,r)$ be a admissible pair.
\begin{itemize}
 \item[(i)] A pair $(s,p)$ is called a $(q,r)$-admissible pair if
\begin{equation}\label{est_dual_loc}
 \frac{2}{s}+\frac{3}{p}<\frac{7}{2},\qquad
\begin{cases}
\frac{1}{2}\leqslant\frac{1}{p}\leqslant 1&2\leqslant r<3\\
\frac{1}{2}\leqslant\frac{1}{p}<\frac{1}{r}+\frac{2}{3}&3\leqslant r\leqslant 6\,.
\end{cases}
\end{equation}
\item[(ii)] A pair $(s,p)$ is a called $(q,r)$-grad-admissible pair if
\begin{equation}\label{est_grad_dual_loc}
\frac{2}{s}+\frac{3}{p}<\frac{5}{2},\quad\frac{1}{2}\leqslant\frac{1}{p}<\frac{1}{r}+\frac{1}{3}\,.
\end{equation}

\end{itemize}
\end{definition}
\begin{remark}\label{re:rel_adm}
If $(s,p)$ is a $(q,r)$-grad-admissible pair, then it is also $(q,r)$-admissible. Moreover, if $(s,p)$ is a $(q,r)$-admissible pair (resp.~$(q,r)$-grad-admissible), and $(q_1,r_1)$ is another admissible pair with $r_1<r$, then $(s,p)$ is also a $(q_1,r_1)$-admissible pair (resp.~ $(q_1,r_1)$-grad-admissible) pair.
\end{remark}
We can state now a useful Corollary to Proposition \ref{le:stri_hs}.

\begin{corollary}\label{co:stri_lo}
Let $\eps>0$ and $T>0$, and let $(q,r)$ be a admissible pair.
\begin{itemize}
 \item[(i)] For any $(q,r)$-admissible pair $(s,p)$,
\begin{equation}\label{eq:st_st}
\begin{split}
 \Big\|\int_{0}^{t} e^{(\ii+\eps)(t-\tau)\Delta}F(\tau)\,\,\ud\tau\,\Big\|_{L^q([0,T],L^r(\R^3))}\;&\lesssim_{\,\eps}\;T^{\theta}\|F\|_{L^s([0,T],L^p(\R^3))} \\
 & \qquad \textstyle \theta:=\frac74-\frac1s-\frac{3}{2p}\,.
\end{split}
\end{equation}
\item[(ii)] Assume in addition that $(q,r)$ is non-endpoint. For any $(q,r)$-grad-admissible pair,
\begin{equation}\label{eq:grst_st}
\begin{split}
\Big\|\nabla\!\!\int_{0}^{t} e^{(\ii+\eps)(t-\tau)\Delta}F(\tau)\,\,\ud\tau\,\Big\|_{L^q([0,T],L^r(\R^3))}\;&\lesssim_{\,\eps}\;T^{\theta}\|F\|_{L^s([0,T],L^p(\R^3))} \\
& \qquad \textstyle \theta:=\frac54-\frac1s-\frac{3}{2p}\,.
\end{split}
\end{equation}
\end{itemize}
In either case, it follows by the assumptions that $\theta>0$.
\end{corollary}

\subsection{Further technical Lemmas}

We conclude the Section by collecting a few technical Lemmas that will be useful for setting up the fixed point argument (Section \ref{sec:propagator}).

Let us first introduce the following.

\begin{definition}\label{relevant}
Given $T>0$, we define 
$$X^{(4,3)}[0,T]\;:=\;L^{\infty}([0,T],H^{1}(\R^3))\cap L^{4}([0,T],W^{1,3}(\R^3))$$
equipped with the Banach norm
$$\|\cdot\|_{X^{(4,3)}[0,T]}:=\|\cdot\|_{L^{\infty}([0,T],H^{1}(\R^3))}+\|\cdot\|_{L^{4}([0,T],W^{1,3}(\R^3))}\,.$$
\end{definition}

\begin{remark}
By interpolation we have that, for every admissible pair $(q, r)$ with $r\in[2, 3]$.
\begin{equation}\label{int_Str}
\|u\|_{L^{q}([0,T],W^{1,r}(\R^3))}\;\lesssim\; \|u\|_{X^{(4,3)}[0,T]}\,.
\end{equation}
Furthermore, Sobolev embedding also yields
\begin{equation}\label{int_Str_emb}
\|u\|_{L^{q}([0,T],L^{\frac{3r}{3-r}}(\R^3))}\;\lesssim\; \|u\|_{X^{(4,3)}[0,T]}
\end{equation}
for any admissible pair $(q, r)$ with $r\in[2, 3]$.
\end{remark}

\begin{lemma}\label{le:abcd}~
\begin{itemize}
\item[(i)] Let $A\in\widetilde{\mathcal A}_1$ or $A\in\widetilde{\mathcal A}_2$. There exist $(4,3)$-grad-admissible pairs  $(s_1,p_1)$, $(s_2,p_2)$ such that, for any $u\in X^{(4,3)}[0,T]$, 
\begin{equation*}
A_i\cdot\nabla u\;\in\; L^{s_i}([0,T],L^{p_i}(\R^3))\,,\qquad i\in\{1,2\}\,,
\end{equation*}
and
\begin{equation}\label{holdgradab}
\begin{aligned}
\|A_i\cdot\nabla u\|_{L^{s_i}([0,T],L^{p_i}(\R^3))}\;\lesssim \;\|A\|_{L^{a_i}([0,T],L^{b_i}(\R^3))}\|u\|_{X^{(4,3)}[0,T]}\,.
\end{aligned}
\end{equation}
\item[(ii)] Let $A\in\widetilde{\mathcal A}_1$. There exist four $(4,3)$-grad-ad\-mis\-sible pairs  $(s_{ij},p_{ij})$, $i,j\in\{1,2\}$, such that, for any $u\in X^{(4,3)}[0,T]$, 
\begin{equation*}
A_i\cdot A_j\,u\;\in\; L^{s_{ij}}([0,T],L^{p_{ij}}(\R^3))
\end{equation*}
and
\begin{equation}\label{holdgradcd}
\begin{aligned}
 &\|A_i\cdot A_j\,u\|_{L^{s_{ij}}([0,T],L^{p_{ij}}(\R^3))}\\
&\quad\lesssim\|A_i\|_{L^{a_i}([0,T],L^{b_i}(\R^3))}\,\|A_j\|_{L^{a_j}([0,T],L^{b_j}(\R^3))}\,\|u\|_{X^{(4,3)}[0,T]}\,.
\end{aligned}
\end{equation}
\item[(iii)] Let $A\in\widetilde{\mathcal A}_2$. There exist four $(4,3)$-admis\-sible pairs $(s_{ij},p_{ij})$, $i,j\in\{1,2\}$, such that, for any $u\in X^{(4,3)}[0,T]$, 
\begin{equation*}
A_i\cdot A_j\,u\;\in\; L^{s_{ij}}([0,T],W^{1,p_{ij}}(\R^3))
\end{equation*}
and
\begin{equation}\label{holdgradcd_grad}
\begin{aligned}
&\|A_i\cdot A_j\,u\|_{L^{s_{ij}}([0,T],W^{1,p_{ij}}(\R^3))}\\
&\quad \lesssim \big(\|A_i\|_{L^{a_i}([0,T],L^{b_i}(\R^3))}+\|\nabla A_i\|_{L^{a_i}([0,T],L^{3b_i/(3+b_i)}(\R^3))}\big)\;\times\\
&\qquad\times\big(\|A_j\|_{L^{a_j}([0,T],L^{b_j}(\R^3))}+\|\nabla A_j\|_{L^{a_j}([0,T],L^{3b_j/(3+b_j)}(\R^3))}\big)\;\times \\
&\qquad\times\|u\|_{X^{(4,3)}[0,T]}\,.
\end{aligned}
\end{equation}
\end{itemize}
\end{lemma}

\begin{proof}
The proof consists in repeatedly applying H\"older's inequality and the Sobolev embedding, we omit the standard details.
\end{proof}

\begin{lemma}\label{le:Vab_grad}
Let $A\in\widetilde{\mathcal A}_1$ or $A\in\widetilde{\mathcal A}_2$, and let $\eps>0$. There exists a constant $\theta_{\!A}>0$ such that, for every $T\in(0,1]$,
\begin{equation*}
\Big\Vert\int_0^te^{(\ii+\eps)(t-\sigma)\Delta}A(\sigma)\cdot\nabla u(\sigma)\,\ud\sigma\,\Big\Vert_{X^{(4,3)}[0,T]}\;\lesssim_{\,\eps,A}\; T^{\theta_{\!A}}\Vert u\Vert_{X^{(4,3)}[0,T]}\,.
\end{equation*}
\end{lemma}

\begin{proof}
Because of Lemma \ref{le:abcd}(i),
\begin{equation*}
A_i\cdot\nabla u\in L^{s_i}([0,T],L^{p_i}(\R^3))\,,\qquad i\in \{1,2\}\,,
\end{equation*}
for some $(s_1,p_1)$, $(s_2,p_2)$ which are $(4,3)$-grad-admissible pairs. Applying Corollary \ref{co:stri_lo}(ii) and Lemma \ref{le:abcd}(i) to $A_i\cdot\nabla u$  and setting 
\begin{equation*}
\theta_{\!A}\;:=\;\min\left\{\frac54-\frac{1}{s_1}-\frac{3}{2p_1},\frac54-\frac{1}{s_2}-\frac{3}{2p_2}\right\}
\end{equation*}
the thesis follows.
\end{proof}

\begin{lemma}\label{le:Vcd_grad}
Let $A\in\widetilde{\mathcal A}_1$ or $A\in\widetilde{\mathcal A}_2$, and let $\eps>0$.
There exists a constant $\theta_{\!A}>0$ such that, for every $T\in(0,1]$,
\begin{equation*}
\left\Vert\int_0^te^{(\ii+\eps)(t-\sigma)\Delta}|A(\sigma)|^2u(\sigma)\,\ud\sigma\right\Vert_{X^{(4,3)}[0,T]}\;\lesssim_{\,\eps,A}\; T^{\theta_{\!A}}\Vert u\Vert_{X^{(4,3)}[0,T]}\,.
\end{equation*}
\end{lemma}

\begin{proof}
The proof is similar to the previous one. For example, in the case $A\in\widetilde{\mathcal A}_1$, by Lemma \ref{le:abcd}(ii) we have
\begin{equation*}
A_i\cdot A_j u\in L^{s_{ij}}([0,T],L^{p_{ij}}(\R^3))\,,\qquad i,j\in\{1,2\}\,.
\end{equation*}
Then we apply Corollary \ref{co:stri_lo}(i).
\end{proof}

\section{The regularised magnetic Laplacian}\label{sec:propagator}

We discuss now the existence of the linear magnetic viscous propagator and we prove that, with our assumptions on the magnetic potential, the propagator enjoys the same Strichartz-type estimates for the heat-Schr\"odinger flow obtained already in the Subsection \ref{sec:smoothingestHSflow}.


The main result of this Section is the following.

\begin{theorem}\label{th:energy_linear}
Assume that $A\in\widetilde{\mathcal{A}}_1$ or $A\in\widetilde{\mathcal{A}}_2$. For given $\mathcal{\tau}\in\mathbb{R}$, $\varepsilon>0$, and $f\in H^1(\mathbb{R}^3)$ 
consider the inhomogeneous Cauchy problem
\begin{equation}\label{eq:applin}
\left\{\begin{aligned}
\ii\,\partial_t u\;=\;&-(1-\ii\,\eps)(\Delta u-2\,\ii\,A\cdot\nabla u-|A|^2u)+F+G \\
u(\tau,\cdot)\;=\;&f \\
\end{aligned}\right.
\end{equation}
and the associated integral equation
\begin{equation}\label{duprima}
\begin{split}
 u&(t,\cdot)\;=\;e^{(\ii+\eps)(t-\tau)\Delta}f \\
 &\quad -\ii\!\int_{\tau}^t e^{(\ii+\eps)(t-\sigma)\Delta}\big((1-\ii\,\eps)(2\,\ii\,A\cdot\nabla u+|A|^2u)(\sigma)+F(\sigma)+G(\sigma)\big)\,\ud\sigma\,,
\end{split}
\end{equation}
where
\begin{itemize}
\item $F\in L^{\widetilde{s}}(\R,W^{1,\widetilde{p}}(\R^3))$ for some pair $(\widetilde{s},\widetilde{p})$ that is  $(4,3)$-admissible pair or satisfies \eqref{est_dual} with $(q,r)=(4,3)$, namely $\frac{2}{\widetilde{s}}+\frac{3}{\widetilde{p}}\leqslant\frac{7}{2}$, $\frac{1}{2}\leqslant\frac{1}{\widetilde{p}}<1$;
\item $G\in L^{s}(\R,L^{p}(\R^3))$, for some pair $(s, p)$ that is $(4, 3)$-grad-admissible or satisfies \eqref{est_grad_dual} with $(q,r)=(4,3)$, namely $\frac{2}{s}+\frac{3}{p}\leqslant\frac{5}{2}$, $\frac{1}{2}\leqslant\frac{1}{p}<\frac{2}{3}$.
\end{itemize}
Then there exists a unique solution $u\in\mathcal{C}([\tau,+\infty),H^1(\R^3))$ to \eqref{duprima}. Moreover, for any $T>\tau$ and for any Strichartz pair $(q,r)$, with $r\in[2,3]$,
\begin{equation}\label{stri_nono_grad}
  \Vert u\Vert_{L^q([\tau,T],W^{1,r}(\R^3))}\;\lesssim_{\,\eps,A,T}\Vert f\Vert_{H^1(\R^3)}+\Vert F\Vert_{L^{s}(\R,L^{p}(\R^3))} +\Vert G\Vert_{L^{\widetilde{s}}(\R,W^{1,\widetilde{p}}(\R^3))}\,.
\end{equation}
\end{theorem}

Theorem \ref{th:energy_linear} shows the existence of a unique solution $u$ to the integral equation \eqref{duprima}. From the assumptions on the magnetic potential and the source terms $F, G$ and by using standard arguments in the theory of evolution equations (see for example \cite{CH}) we may also infer that $u$ satisfies \eqref{eq:applin} for almost every $t\in\R$ in the sense of distributions. In the case when $F=G=0$, the solution $u$ to \eqref{eq:applin} defines an evolution operator, namely for any $f\in H^1(\R^3)$ the \emph{magnetic viscous evolution} is defined by $\mathcal U_{\eps, A}(t, \tau)f=u(t)$ where $u$ is the solution to \eqref{eq:applin} with $F=G=0$. As a consequence of Theorem \ref{th:energy_linear} we have that $\mathcal U_{\eps, A}(t, \tau)$ enjoys a class of Strichartz-type estimates.

\begin{proposition}\label{prop:prop}
The family $\{\mathcal U_{\eps, A}(t, \tau)\}_{t,\tau}$ of operators  on $H^1(\mathbb{R}^3)$  satisfies the following properties:
\begin{itemize}
\item $\mathcal U_{\eps, A}(t, s)\,\mathcal U_{\eps, A}(s, \tau)=\mathcal U_{\eps, A}(t, \tau)$ for any $\tau<s<t$;
\item $\mathcal U_{\eps, A}(t, t)=\mathbbm{1}$;
\item the map $(t,\tau)\mapsto\mathcal U_{\eps, A}(t, \tau)$ is strongly continuous in $H^1(\mathbb{R}^3)$;
\item for any admissible pair $(q, r)$ with $r\in[2, 3]$, and for any $F,G$ satisfying the same assumptions as in Theorem \ref{th:energy_linear}, one has
\begin{eqnarray}
 \Vert\mathcal{U}_{\eps,A}(t,\tau)f\Vert_{L^q([\tau,T],W^{1,r}(\R^3))}\!\!\!&\lesssim_{\,\eps,A,T}&\!\!\!\Vert f\Vert_{H^1(\R^3)} \label{eq:mahostg} \\
 \Big\Vert\int_\tau^t\!\mathcal{U}_{\eps,A}(t,\sigma)\,F(\sigma)\,\ud\sigma\,\Big\Vert_{L^q([\tau,T],W^{1,r}(\R^3))}\!\!\!&\lesssim_{\,\eps,A,T}&\!\!\!\Vert F\Vert_{L^{\widetilde{s}}([\tau,T],W^{1,\widetilde{p}}(\R^3))} \label{eq:mainhostg} \\
 \Big\Vert\int_\tau^t\!\mathcal{U}_{\eps,A}(t,\sigma)\,G(\sigma)\,\ud\sigma\,\Big\Vert_{L^q([\tau,T],W^{1,r}(\R^3))}\!\!\!&\lesssim_{\,\eps,A,T}&\!\!\!\Vert G\Vert_{L^s([\tau,T],L^{p}(\R^3))}\,. \label{eq:mainhostg2}
\end{eqnarray}
\end{itemize}
\end{proposition}

%
%

Once we defined the magnetic viscous evolution operator $\mathcal U_{\eps, A}(t, \tau)$, we see that we can write the integral formulation for \eqref{eq:applin} in the following way
\begin{equation}\label{dusecondapp}
u(t)\;=\;\mathcal U_{\eps, A}(t, \tau)f-\ii\int_\tau^t\mathcal U_{\eps, A}(t, \sigma)\big(F(\sigma)+G(\sigma)\big)\,\ud \sigma\,.
\end{equation}
We will use formula \eqref{dusecondapp} and the Strichartz-type estimates \eqref{eq:mahostg}-\eqref{eq:mainhostg2} in order to set up a fixed point argument and show the existence of solutions to the nonlinear problem \eqref{eq:visc_CauMNLS}.

Let us now proceed with proving Theorem \ref{th:energy_linear}.
As already mentioned, the proof is based upon a contraction argument in the space introduced in Definition \ref{relevant} and requires the magnetic estimates established Lemmas \ref{le:abcd}, \ref{le:Vab_grad}, and \ref{le:Vcd_grad}.

\begin{proof}[Proof of Theorem \ref{th:energy_linear}]
It is clearly not restrictive to set the initial time $\tau=0$. For given $T\in(0,1]$ and $M>0$, we consider the ball of radius $M$ in $X^{(4, 3)}[0, T]$, i.e.,
$$\mathcal{X}_{T,M}\;:=\;\{u\in X^{(4,3)}[0,T]\;|\;\Vert u\Vert_{X^{(4,3)}[0,T]}\leqslant M\}.$$
Moreover, we define the solution map $u\mapsto\Phi u$ where, for $t\in[0,T]$,
\begin{equation}\label{eq:solution_map}
\begin{split}
 &(\Phi u)(t)\;:=\;e^{(\ii+\eps)t\Delta}f \\
 &\;-(\ii+\eps)\int_{0}^te^{(\ii+\eps)(t-\sigma)\Delta}\big((2\,\ii\,A(\sigma)\cdot\nabla+|A(\sigma)|^2)u(\sigma)+F(\sigma)+G(\sigma)\big)\,\ud\sigma\,.
\end{split}
\end{equation}
Thus, finding a solution to the integral equation \eqref{duprima}, with $\tau=0$, is equivalent to finding a fixed point for the map $\Phi$. We shall then prove Theorem \ref{th:energy_linear} by showing that, for suitable $T$ and $M$, the map $\Phi$ is a contraction on $\mathcal{X}_{T,M}$. To this aim, let us consider a generic $u\in \mathcal{X}_{T,M}$: owing to the Strichartz estimates \eqref{eq:homstrich} and \eqref{eq:grst} and to Lemmas \ref{le:Vab_grad} and \ref{le:Vcd_grad}, there exist positive constants $C\equiv C_{\eps,A}$ and $\theta\equiv\theta_{\!A}$ such that, for $T\in (0,1]$,
\begin{equation}\label{eq:bou_phi_g}
\begin{split}
 \Vert\Phi u\Vert_{X^{(4,3)}[0,T]}\;&\leqslant\; C\,\Big(\Vert f\Vert_{H^1(\R^3)}+\sum_{i=1}^N\|F_i\|_{L^{s_i}([0,T],L^{p_i}(\R^3)} \\
 &\qquad\qquad+\sum_{i=1}^N\|G_i\|_{L^{\widetilde{s}_i}([0,T],L^{\widetilde{p}_i}(\R^3)} +T^{\theta}\Vert u\Vert_{X^{(4,3)}[0,T]}\Big)\,.
\end{split}
\end{equation}
It is possible to restrict further $M$ and $T$ such that
\[
 M\;>\;2C\,\Big(\Vert f\Vert_{H^1(\R^3)}+\sum_{i=1}^N\|F_i\|_{L^{s_i}([0,T],L^{p_i}(\R^3)}+\sum_{i=1}^N\|G_i\|_{L^{\widetilde{s}_i}([0,T],L^{\widetilde{p}_i}(\R^3)}\Big)
\]
and $2CT^{\theta}<1$, in which case \eqref{eq:bou_phi_g} yields
$$ \Vert\Phi u\Vert_{X^{(4,3)}[0,T]}\;\leqslant\; M({\textstyle\frac{1}{2}}+CT^{\theta})\;<\;M\,.$$
This proves that $\Phi$ maps indeed $\mathcal{X}_{T,M}$ into itself. Next, for generic $u,v\in \mathcal{X}_{T,M}$, and with the above choice of $M$ and $T$, \eqref{eq:bou_phi_g} also yields
\begin{align*}
\Vert \Phi u-\Phi v\Vert_{X^{(4,3)}[0,T]}\;&=\;\Vert \Phi (u-v)\Vert_{X^{(4,3)}[0,T]}\;\leqslant\; CT^{\theta}\Vert u-v\Vert_{X^{(4,3)}[0,T]}\\
&<\;\frac12\Vert u-v\Vert_{X^{(4,3)}[0,T]}\,,
\end{align*}
which proves that $\Phi$ is indeed a contraction on $\mathcal{X}_{T,M}$. By Banach's fixed point theorem, we conclude that the integral equation $u=\Phi u$ has a unique solution in $\mathcal{X}_{T,M}$. Furthermore, $\Phi u\in \mathcal{C}([0,T],H^1(\R^3))$. Hence, we have found a local solution $u\in\mathcal{C}([0,T],H^1(\R^3))$ to the integral equation \eqref{duprima}, which satisfies \eqref{stri_nono_grad}. Moreover, since the local existence time $T$ does not depend on the initial data, this solution can be extended globally in time, and \eqref{stri_nono_grad} is satisfied for any $T>0$.
\end{proof}

As the last result of this Section, we show the propagator $\mathcal{U}_{\eps,A}(t,\tau)$ is stable under small perturbations of the magnetic potential and of the initial datum.

\begin{proposition}[Stability]\label{pr:stab_td}~
Let $\tau\in\R$, $T>\tau$, and let us assume that $A^{(1)}, A^{(2)}\in\widetilde{\mathcal{A}}_1$, with $\|A^{(1)}-A^{(2)}\|_{\mathcal A_1}<\delta$ or $A^{(1)}, A^{(2)}\in\widetilde{\mathcal{A}}_2$, with $\|A^{(1)}-A^{(2)}\|_{\mathcal A_2}<\delta$, where $\delta>0$ is sufficiently small. Let $u_1, u_2\in\mathcal C([\tau, T);H^1(\R^3))$
be the solutions to
\begin{equation}\label{eq:stab_td}
\left\{\begin{aligned}
\ii\,\partial_t u_j\;=\;&-(1-\ii\,\eps)(\nabla-\ii A^{(j)})^2u_j+F_j \\
u(\tau,\cdot)\;=\;&\:f_j \\
\end{aligned}\right.
\end{equation}
for given $f_1, f_2\in H^1(\R^3)$ and given $F_1, F_2\in L^s([\tau,T],W^{1,p}(\R^3))$, where $(s,p)$ is dual-admissible. Then, for any admissible pair $(q,r)$ with $r\in [2,3]$ we have
\begin{equation*}
\|u_1-u_2\|_{L^q([\tau, T],W^{1, r}(\R^3))}\;\lesssim\;\delta+\|f_1-f_2\|_{H^1}+\|F_1-F_2\|_{L^s([\tau,T],W^{1,p}(\R^3))}\,.
\end{equation*}
\end{proposition}

\begin{proof} We prove the Proposition under the assumptions $A^{(1)}, A^{(2)}\in\mathcal A_1$ and $\|A^{(1)}-A^{(2)}\|_{\mathcal A_1}<\delta$, the other case being completely analogous.
From \eqref{eq:stab_td} we infer that the function $\widetilde u:=u_1-u_2$ satisfies
\begin{equation*}
\left\{\begin{aligned}
\ii\,\partial_t \widetilde u\;=\;&-(1-\ii\,\eps)(\nabla-iA^{(1)})^2\widetilde u+2\,\ii\widetilde A\cdot\nabla u_2+\widetilde A\cdot(A^{(1)}+A^{(2)})u_2+\widetilde F \\
\widetilde u(\tau,\cdot)\;=\;&\:\widetilde f \\
\end{aligned}\right.
\end{equation*}
or equivalently
\begin{equation}\label{eq:diff}
\begin{split}
 \widetilde u(t)\;&=\;\mathcal U_{\eps, A^{(1)}}(t,0)\widetilde f \\
 &\quad -\ii\int_\tau^t\mathcal U_{\eps, A^{(1)}}(t, \sigma)\big(2\,\ii\widetilde A\cdot\nabla u_2+\widetilde A\cdot(A^{(1)}+A^{(2)})u_2+\widetilde F\big)(\sigma)\,\ud \sigma\,,
\end{split}
\end{equation}
where $\widetilde f:=f_1-f_2$, $\widetilde A:=A_1-A_2$, and $\widetilde F:=F_1-F_2$. Since $u_1$ and $u_2$ solve \eqref{eq:stab_td} on the time interval $[\tau, T]$, estimate \eqref{stri_nono_grad}  yields
\begin{equation*}
\|u_j\|_{L^q([\tau, T];W^{1, r}(\R^3))}\;\leqslant\; C(\|f_j\|_{H^1}, \|A^{(j)}\|_{\mathcal A_1}, \|F_j\|_{L^s([\tau,T],W^{1,p}(\R^3))})\,,\quad j\in\{1, 2\}
\end{equation*}
for any admissible pair $(q, r)$ with $r\in[2, 3]$. 
By applying the Strichartz-type estimates stated in Proposition \ref{prop:prop} and the estimates of  Lemma \ref{le:abcd} to equation \eqref{eq:diff} we have
\begin{equation*}
\begin{aligned}
\|\tilde u&\|_{L^q([\tau, T],W^{1, r}(\R^3))}\;\lesssim\;\|\widetilde{f}\|_{H^1}+\|\widetilde{A}\|_{\mathcal A_1}\|u_2\|_{X^{(4, 3)}[\tau, T]}\\
&+\|\widetilde{A}\|_{\mathcal A_1}\left(\|A^{(1)}\|_{\mathcal A_1}+\|A^{(2)}\|_{\mathcal A_1}\right)\|u_2\|_{X^{(4, 3)}[\tau, T]}+\|\tilde F\|_{L^s([\tau,T],W^{1,p}(\R^3))}\,,
\end{aligned}
\end{equation*}
from which the result follows.
\end{proof}

\section{Local well posedness for the regularised magnetic NLS}\label{sec:LWP}

In this Section we turn our attention to the non-linear problem \eqref{eq:visc_CauMNLS}. Using the existence result and the Strichartz-type estimates established, respectively, in Theorem \ref{th:energy_linear} and Proposition \ref{prop:prop}, we set up our fixed point argument associated with the integral equation
\begin{equation}\label{eq:visc_nls_int}
u(t)\;=\;\mathcal{U}_{\eps,A}(t,0)f-\ii\!\int_0^t\mathcal{U}_{\eps,A}(t,\sigma)\,\mathcal{N}(u)(\sigma)\,\ud\sigma\,.
\end{equation}

We first focus on the case of energy sub-critical non-linearities.

\begin{proposition}[Local well-posedness, energy sub-critical case]\label{pr:exsub}
Let $\eps>0$. Assume that $A\in\widetilde{\mathcal A}_1$ or $A\in\widetilde{\mathcal A}_2$ and that the exponents in the non-linearity \eqref{hart_pp_nl} are in the regime $\gamma\in(1,5)$ and $\alpha\in(0,3)$. Then for any $f\in H^1(\mathbb{R}^3)$ there exists a unique solution $u\in\mathcal C([0, T_{max}),H^1(\mathbb{R}^3))$ to \eqref{eq:visc_nls_int} on a maximal interval $[0, T_{max})$ such that the following blow-up alternative holds: if $T_{max}<+\infty$ then $\lim_{t\uparrow T_{max}}\|u(t)\|_{H^1}=+\infty$. 
\end{proposition}

\begin{proof}
Since the linear propagator $\mathcal U_{\eps, A}(t, \tau)$ satisfies the same Strichartz-type estimates as the heat-Schr\"odinger flow, and since the non-linearities considered here are sub-critical perturbation of the linear flow, a customary contraction argument in the space 
\begin{equation}\label{eq:cont_space}
\mathcal{C}([0,T],H^1(\R^3))\cap L^{q(\gamma)}([0,T],W^{1,r(\gamma)}(\R^3))\cap  L^{q(\alpha)}([0,T],W^{1,r(\alpha)}(\R^3))\,,
\end{equation}
where
\begin{equation}
(q(\gamma),r(\gamma))\;:=\;\textstyle\Big(\frac{4(\gamma+1)}{\gamma-1},\frac{3(\gamma+1)}{\gamma+2}\Big)
\end{equation}
(see, e.g., \cite[Theorems 2.1 and 3.1]{Miao-Hartree-2007}) 
and
\begin{equation}\label{qalpha}
(q(\alpha),r(\alpha))\;:=\; \begin{cases}
\quad\;(+\infty,2)&\;\;\alpha\in(0,2]\\
\Big(\frac{6}{\alpha-2},\frac{18}{13-2\alpha}\Big)&\;\;\alpha\in(2,3)\\
\end{cases}
\end{equation}
(see, e.g., \cite[Section 5.2]{Lineares-Ponce-book2015}),
guarantees the existence of a unique local solution for sufficiently small $T$. We observe, in particular, that with the above choice one has $r(\gamma),r(\alpha)\in[2,3)$. Furthermore, by a customary continuation argument we can extend such a solution over a maximal interval for which the blow-up alternative holds true. We omit the standard details, they are part of the well-established theory of semi-linear equations.
\end{proof}

In the presence of a energy-critical non-linearity ($\gamma=5$) the above arguments cannot be applied. Indeed, when $\gamma=5$ we cannot apply Corollary \ref{co:stri_lo} with that nonlinearity, in order to obtain the factor $T^\theta$, $\theta>0$ and apply the standard contraction argument. However, it is possible to exploit a similar idea as in \cite{Cazenave-Weissler-1990} to infer a local well-posedness result when $\gamma=5$.

\begin{proposition}[Local existence and uniqueness, energy critical case]\label{pr:crie}
Let $A\in\widetilde{\mathcal A}_1$ or $A\in\widetilde{\mathcal A}_2$ and let the exponents in the non-linearity \eqref{hart_pp_nl} be in the regime $\gamma=5$ and $\alpha\in(0,3)$. Let $\eps>0$ and $f\in H^1(\R^3)$. There exists $\eta_0>0$ such that, if 
\begin{equation}
\|\nabla e^{\ii t\Delta}f\|_{L^6([0,T],L^{\frac{18}{7}}(\R^3))}\leqslant \eta
\end{equation}
for some (small enough) $T>0$ and some $\eta<\eta_0$, then there exists a unique solution $u\in\mathcal C([0, T],H^1(\mathbb{R}^3))$ to \eqref{eq:visc_nls_int}.
Moreover, this solution can be extended on a maximal interval $[0, T_{max})$ such that the following blow-up alternative holds true: $T_{max}<\infty$ if and only if $\Vert u\Vert_{L^6([0,T_{max}),L^{18}({R}^3))}=\infty$.
\end{proposition}

\begin{proof}
A direct application of a well-known argument by Cazenave and Weissler \cite{Cazenave-Weissler-1990} (we refer to \cite[~Section 3]{Killip-Visan-2013} for a more recent discussion). In particular, having established Strichartz estimates for $\mathcal{U}_{\eps,A}(t,\tau)$ relative to the pair $(q,r)=(6,\frac{18}{7})$, we proceed exactly as in the proof of \cite[Theorem 3.4 and Corollary 3.5]{Killip-Visan-2013},  so as to find a unique solution $u$ to the integral equation \eqref{eq:visc_nls_int} in the space
\begin{equation}\label{eq:cont_space_crit}
\mathcal{C}([0,T],H^1(\R^3))\cap L^{6}([0,T],W^{1,\frac{18}{7}}(\R^3))\cap  L^{q(\alpha)}([0,T],W^{1,r(\alpha)}(\R^3))
\end{equation}
with $(q(\alpha),r(\alpha))$ given by \eqref{qalpha}, together with the $L^6_tL^{18}_x$-blow-up alternative.
\end{proof}

We conclude this Section by stating the analogous stability property of Proposition \ref{pr:stab_td} also for the nonlinear problem
\begin{proposition}\label{prop:stab_nl}
Let $\tau\geq0$, $T\in(\tau, \infty)$ and let us assume that $A^{(1)}, A^{(2)}\in \widetilde{\mathcal{A}}_1$ with $\|A^{(1)}-A^{(2)}\|_{\mathcal A_1}<\delta$ 
or that  $A^{(1)},A^{(2)}\in \widetilde{\mathcal{A}}_2$ with $\|A^{(1)}-A^{(2)}\|_{\mathcal A_2}<\delta$, for some $\delta>0$ small enough. Let us consider $u_1, u_2\in\mathcal C([\tau, T];H^1(\R^3))$ solutions to 
\begin{equation*}
\left\{\begin{aligned}
\ii\,\partial_t u_j\;=\;&-(1-\ii\,\eps)(\nabla-iA^{(j)})^2u_j+\mathcal N(u_j) \\
u(\tau,\cdot)\;=\;&f_j, \\
\end{aligned}\right.
\end{equation*}
where $j\in\{1, 2\}$, $f_1, f_2\in H^1$, $\mathcal N(u)$ is given by \eqref{hart_pp_nl} with $\gamma\in(1, 5], \alpha\in(0, 3)$. Then for any admissible pair $(q, r)$ with $r\in[2, 3]$ we have
\begin{equation*}
\|u_1-u_2\|_{L^q([\tau, T],W^{1, r}(\R^3))}\;\lesssim\;\delta+\|f_1-f_2\|_{H^1}.
\end{equation*}
\end{proposition}

\section{Mass and energy estimates}\label{eq:mass-energy}

In this Section we establish some a priori estimates which will be needed in order to extend the local approximating solution obtained in Section \ref{sec:LWP} over arbitrary time intervals. In particular we will show that the total mass and energy are uniformly bounded. Furthermore, by exploiting the dissipative regularisation, we will infer some a priori space-time bounds which will allow to extend globally the solution also in the energy-critical case.



The two quantities of interest are defined as follows.

\begin{definition}
Let $T>0$. For each $u\in L^{\infty}([0,T),H^1(\R^3))$ and $t\in [0,T)$, mass and energy of $u$ are defined, at almost every time $t\in[0,T)$, as
\[
 \begin{split}
  (\mathcal{M}(u))(t)\;&:=\;\int_{\R^3}|u(t,x)|^2\,\ud x \\
  (\mathcal{E}(u))(t)\;&:=\;\int_{\R^3}\Big({\textstyle{\frac{1}{2}}}|(\nabla-\ii A(t))\,u|^2+{\textstyle\frac{1}{\gamma +1}}|u|^{\gamma +1}+{\textstyle{\frac{1}{4}}}(|x|^{-\alpha}*|u|^2)|u|^2\Big)\,\ud x\,.
 \end{split}
\]
\end{definition}

In what follows, we will consider potentials $A\in\mathcal A_1$ or $A\in\mathcal A_2$, so to have the time regularity needed in order to study the energy functional.

\begin{proposition}\label{pr:5ene}
Assume that $A\in\mathcal{A}_1$ or $A\in\mathcal{A}_2$, and that the exponents in the non-linearity \eqref{hart_pp_nl} are in the whole regime $\gamma\in(1,5]$ and $\alpha\in(0,3)$. For fixed $\eps>0$, let $u_{\eps}\in\mathcal{C}([0,T),H^1(\R^3))$ be the local solution to the regularised equation \eqref{eq:visc_nls} for some $T>0$. Then the mass, the energy, and the $H^1$-norm of $u_{\eps}$ are bounded in time over $[0,T)$, uniformly in $\eps>0$, that is,
\begin{eqnarray}
 \sup_{t\in[0, T]}\mathcal{M}(u_{\eps})\!\!\!&\!\!\!\!\!\!\!\!\!\lesssim & \!\!\!\!1\label{eq:Mbound} \\
 \sup_{t\in[0, T]}\mathcal{E}(u_{\eps})\!\!&\lesssim_{A,T} &\!\!\!\! 1 \label{eq:Ebound}\\
 \Vert u_{\eps}\Vert_{L^{\infty}([0,T),H^1(\R^3))}\!\!\!&\lesssim_{A,T}&\!\!\!\!1\,, \label{eq:uni_bou}
\end{eqnarray}
and moreover one has the a priori bounds
\begin{equation}\label{dissiesti}
\begin{aligned}
\int_0^{T}\!\!\int_{\R^3}&\Big(|(\nabla-\ii A(t))u_{\eps}|^2\big(|u_{\eps}|^{\gamma-1}+(|x|^{-\alpha}*|u_{\eps}|^2)\big)+(\gamma-1)|u_{\eps}|^{\gamma-1}|\nabla|u_{\eps}||^2\\
&\qquad+(|x|^{-\alpha}*\nabla|u_{\eps}|^2)\nabla|u_{\eps}|^2\Big)\,\ud x\,\ud t\;\lesssim_{A,T}\;\varepsilon^{-1}\,.
\end{aligned}
\end{equation}
\end{proposition}

\begin{remark}
  \emph{At fixed} $\varepsilon>0$ the finiteness of $\mathcal{M}(u_{\eps})(t)$ and of $\mathcal{E}(u_{\eps})(t)$ for all $t\in[0,T)$ is obvious for the mass, since by assumption $u_{\eps}(t)\in L^2(\R^3)$ for every $t\in[0,T)$, and it is also straightforward for the energy, since the property that $((\nabla-\ii A)u_\varepsilon)(t)\in L^2(\R^3)$ for every $t\in[0,T)$ is also part of the assumption, and moreover it is a standard property (see, e.g., \cite[Section 3.2]{cazenave}) that both $\int_{\mathbb{R}^3}|u_\varepsilon|^{\gamma +1}\,\ud x$ and $\int_{\mathbb{R}^3}(|x|^{-\alpha}*|u_\varepsilon|^2)|u_\varepsilon|^2\,\ud x$ are finite for every $t\in[0,T)$, and both in the energy sub-critical and  critical regime. 
  The virtue of Proposition \ref{pr:5ene} is thus to produce bounds \eqref{eq:Mbound}-\eqref{eq:uni_bou} that are \emph{uniform} in $\varepsilon$. The non-uniformity in $T$ of \eqref{eq:Ebound}-\eqref{eq:uni_bou} is due to the fact that the magnetic potential is only $AC_\mathrm{loc}$ in time: for $AC$-potentials such bounds would be uniform in $T$ as well.
\end{remark}

\begin{proof}[Proof of Proposition \ref{pr:5ene}]
We recall that $u_{\eps}$ satisfies 
$$\ii\,\partial_t u_{\eps}\;=\;-(1-\ii\,\eps)(\nabla-\ii\,A)^2u_{\eps}+\mathcal{N}(u_{\eps})$$ as an identity at every $t$ between $H^{-1}$-functions in space.

Let us first prove the thesis in a regular case, and later work out a density argument for the general case.

It is straightforward to see, by means of a customary contraction argument in $L^\infty([0,T],H^s(\mathbb{R}^3))$ for arbitrary $s>0$, that if $f\in\mathcal{S}(\R^3)$ and $A\in AC_{\mathrm{loc}}(\R,\mathcal{S}(\R^3))$, then
the solution $u_{\eps}$ to the local Cauchy problem \eqref{eq:visc_CauMNLS} is smooth in space, whence in particular $u_{\eps}\in \mathcal{C}^1([0,T),H^1(\R^3))$, a fact that justifies the time derivations in the computations that follow. 

From
\[
 \begin{split}
  \frac{\ud}{\ud t}&(\mathcal{M}(u_{\eps}))(t)\;= \\
  &=\;-2\,\mathfrak{Re}\int_{\R^3}\overline{u_{\eps}}\,\big((\ii+\eps)(\nabla-\ii\,A)^2u_{\eps}-\ii\,|u_{\eps}|^{\gamma-1}u_{\eps}-\ii\,(|\cdot|^{-\alpha}*|u_{\eps}|^2)\,u_{\eps}\big)\,\ud x \\
  &=\;-2\eps\int_{\R^3}|(\nabla-\ii\,A)u_{\eps}|^2dx\;\leqslant\; 0\,,
 \end{split}
\]
one deduces $(\mathcal{M}(u_{\eps}))(t)\leqslant (\mathcal{M}(u_{\eps}))(0)$, whence \eqref{eq:Mbound}. 

Next, we compute
\begin{equation}\label{derieni}
\begin{aligned}
&\frac{\ud}{\ud t}(\mathcal{E}(u_{\eps}))(t)\;=\;\mathfrak{Re}\int_{\R^3}\Big(\big((\nabla-\ii\,A)\partial_t u_{\eps}-\ii\,(\partial_tA)u_{\eps}\big)\cdot\overline{(\nabla-\ii\,A)u_{\eps}}\\
&\qquad\qquad\qquad\qquad\qquad\qquad+\big(|u_{\eps}|^{\gamma -1}+(|x|^{-\alpha}*|u_{\eps}|^2)\big)\,\overline{u_{\eps}}\,\partial_t u_{\eps}\Big)\,\ud x\\
&=\;\mathfrak{Re}\int_{\R^3}(\partial_t u_{\eps})\big(-\overline{(\nabla-\ii\,A)^2u_{\eps}}+|u_{\eps}|^{\gamma-1}\,\overline{u_{\eps}}+(|x|^{-\alpha}*|u_{\eps}|^2)\,\overline{u_{\eps}}\,\big)\,\ud x\\
&\qquad\qquad +\int_{\R^3}A\cdot(\partial_tA)|u_{\eps}|^2+(\partial_tA)\cdot\mathfrak{Im}\,(u_{\eps}\overline{\nabla  u_{\eps}}\,)\,\ud x\\
&=\;\eps\!\int_{\R^3}\!\big(\!-|(\nabla-\ii\,A)^2u_{\eps}|^2+(\,|u_{\eps}|^{\gamma-1}+|x|^{-\alpha}*|u_{\eps}|^2)\,\mathfrak{Re}\,(\overline{u_{\eps}}(\nabla-\ii\,A)^2u_{\eps})\big)\,\ud x\!\!\!\!\!\!\!\\
&\qquad\qquad +\int_{\R^3}\!\big(A\cdot(\partial_tA)|u_{\eps}|^2+(\partial_tA)\cdot\mathfrak{Im}\,(u_{\eps}\overline{\nabla  u_{\eps}}\,)\,\ud x \\
&=\;-\eps\!\int_{\R^3}\!\,|(\nabla-\ii\,A)^2u_{\eps}|^2\,\ud x -\varepsilon\,\mathcal{R}(u_\varepsilon)(t)+\mathcal{S}(u_\varepsilon)(t)\,,
\end{aligned}
\end{equation}
where
\begin{equation*}
 \begin{split}
  \mathcal{R}(u_\varepsilon)(t)\;&:=\;-\!\int_{\R^3}\!(\,|u_{\eps}|^{\gamma-1}+|x|^{-\alpha}*|u_{\eps}|^2)\,\mathfrak{Re}\,(\overline{u_{\eps}}(\nabla-\ii\,A)^2u_{\eps})\,\ud x \\
  \mathcal{S}(u_\varepsilon)(t)\;&:=\;\int_{\R^3}\!\big(A\cdot(\partial_tA)|u_{\eps}|^2+(\partial_tA)\cdot\mathfrak{Im}\,(u_{\eps}\overline{\nabla  u_{\eps}}\,)\,\big)\,\ud x\,.
 \end{split}
\end{equation*}

From
\begin{equation}\label{dopoderieni}
\begin{aligned}
&\mathcal{R}(u_\varepsilon)(t)\;=\\
&=\;-\int_{\R^3}(\,|u_{\eps}|^{\gamma-1}+|x|^{-\alpha}*|u_{\eps}|^2)\big(-|(\nabla-\ii\,A)u_{\eps}|^2+{\textstyle{\frac{1}{2}}}\Delta|u_{\eps}|^2\big)\,\ud x\\
&=\;+\!\int_{\R^3}|u_{\eps}|^{\gamma-1}|(\nabla-\ii\,A)u_{\eps}|^2\,\ud x+(\gamma-1)\!\int_{\R^3}|u_{\eps}|^{\gamma-1}|\nabla|u_{\eps}||^2\,\ud x\\
&\quad\;\, +\!\int_{\R^3}(\,|x|^{-\alpha}*|u_{\eps}|^2)\,|(\nabla-\ii\,A)u_{\eps}|^2\,\ud x+{\textstyle{\frac{1}{2}}}\!\int_{\R^3}(\,|x|^{-\alpha}*\nabla|u_{\eps}|^2)\,\nabla|u_{\eps}|^2\,\ud x\!\!\!\!\!\!\!
\end{aligned}
\end{equation}
we see that
\begin{equation}\label{eq:Rpositive}
 \mathcal{R}(u_\varepsilon)(t)\;\geqslant\;0\,.
\end{equation}
This is obvious for the first three summands in the r.h.s.~of \eqref{dopoderieni}, whereas for the last one, setting $\phi:=\nabla|u_{\eps}|^2$, Plancherel's formula gives
$$\int_{\R^3}(\,|x|^{-\alpha}*\phi)\,\phi\,\ud x\;=\;\int_{\R^3}\widehat{(|\cdot|^{-\alpha})}(\xi)\,|\widehat{\phi}(\xi)|^2\,\ud \xi\,,$$
and since $\widehat{|\cdot|^{-\alpha}}$ is positive, the fourth summand too is positive.
Therefore, 
\begin{equation}\label{enel1}
\frac{\ud}{\ud t}(\mathcal{E}(u_{\eps}))(t)\;\leqslant\;\mathcal{S}(u_\varepsilon)(t)\,.
\end{equation}

In order to estimate $\mathcal{S}(u_\varepsilon)(t)$,
it is checked by direct inspection that there are $M_1,M_2\in[2,6]$ such that
\[
 b_1*2*M_1\;=\;b_2*2*M_2\;=\;1\,,
\]
whence, for every $t\in[0,T)$ and $j\in\{1,2\}$,
\[
 \begin{split}
  \|u_\varepsilon(t)\|_{L^M_j(\mathbb{R}^3)}\;&\lesssim\; \|u_\varepsilon(t)\|_{H^1(\mathbb{R}^3)} \\
  &\lesssim\;\big(1+ \Vert A_1(t)\Vert_{L^{b_1}(\R^3)}+ \Vert A_2(t)\Vert_{L^{b_2}(\R^3)}\big)\,\Vert u_{\eps}\Vert_{H^1_{A(t)}}
 \end{split}
\]
(Sobolev's embedding and norm equivalence \eqref{eq:norm_equiv}). Thus, by H\"{o}lder's inequality,
\begin{equation}\label{enel2}
\begin{aligned}
\Big|\int_{\R^3}&(\partial_t A(t))\cdot \mathfrak{Im}\,(u_{\eps}(t)\,\overline{\nabla u_{\eps}(t)})\,\ud x\,\Big|\\
&\lesssim\; \big(\Vert\partial_t A_1(t)\Vert_{L^{b_1}(\R^3)}+\Vert\partial_t A_2(t)\Vert_{L^{b_2}(\R^3)}\big)\:\times\\
&\qquad\quad \times\big(1+\Vert A_1(t)\Vert_{L^{b_1}(\R^3)}+ \Vert A_2(t)\Vert_{L^{b_2}(\R^3)}\big)^2\,\Vert u_{\eps}(t)\Vert_{H_{A(t)}^1(\R^3)}^2\\
&\leqslant\;\big(\Vert\partial_t A_1(t)\Vert_{L^{b_1}(\R^3)}+\Vert\partial_t A_2(t)\Vert_{L^{b_2}(\R^3)}\big)\:\times\\
&\qquad\quad \times\big(1+\Vert A_1(t)\Vert_{L^{b_1}(\R^3)}+ \Vert A_2(t)\Vert_{L^{b_2}(\R^3)}\big)^2\,\big(1+(\mathcal{E}(u_{\eps})(t)\big)\,,
\end{aligned}
\end{equation}
the last step following from 
\begin{equation}\label{pdelc}
 \Vert u_{\eps}(t)\Vert_{H_{A(t)}^1}^2\leqslant\; (\mathcal{M}(u_{\eps}))(t)+(\mathcal{E}(u_{\eps}))(t)
\end{equation}
and from $(\mathcal{M}(u_{\eps}))(t)\lesssim 1$. Analogously, now with H\"{o}lder exponents $M_{ij}\in[2,6]$ such that 
\[
 b_i*b_j*{\textstyle\frac{1}{2} }M_{ij}\;=\;1\qquad i,j\in\{1,2\}\,,
\]
we find
\begin{equation}\label{enel3}
\begin{aligned}
\Big|\int_{\R^3}&A\cdot(\partial_tA)\,|u_{\eps}|^2\,\ud x\,\Big|\\
&\lesssim\; \big(\Vert\partial_t A_1(t)\Vert_{L^{b_1}(\R^3)}+\Vert\partial_t A_2(t)\Vert_{L^{b_2}(\R^3)}\big)\:\times\\
&\qquad\quad \times\big(\Vert A_1(t)\Vert_{L^{b_1}(\R^3)}+ \Vert A_2(t)\Vert_{L^{b_2}(\R^3)}\big)\,\Vert u_{\eps}(t)\Vert_{H_{A(t)}^1(\R^3)}^2\\
&\leqslant\;\big(\Vert\partial_t A_1(t)\Vert_{L^{b_1}(\R^3)}+\Vert\partial_t A_2(t)\Vert_{L^{b_2}(\R^3)}\big)\:\times\\
&\qquad\quad \times\big(1+\Vert A_1(t)\Vert_{L^{b_1}(\R^3)}+ \Vert A_2(t)\Vert_{L^{b_2}(\R^3)}\big)\,\big(1+(\mathcal{E}(u_{\eps})(t)\big)\,.
\end{aligned}
\end{equation}

Combining \eqref{enel1}, \eqref{enel2} and \eqref{enel3} together yields
\begin{equation}\label{enel4}
 \begin{split}
  \frac{\ud}{\ud t}(\mathcal{E}(u_{\eps}))(t)\;&\lesssim\;|\mathcal{S}(u_\varepsilon)(t)|\;\lesssim\; \Lambda(t)\,\big(1+(\mathcal{E}(u_{\eps})(t)\big) \\
  \Lambda(t)\;&:=\;\big(\Vert\partial_t A_1(t)\Vert_{L^{b_1}(\R^3)}+\Vert\partial_t A_2(t)\Vert_{L^{b_2}(\R^3)}\big)\:\times \\
  &\qquad\quad\times\:\big(1+\Vert A_1(t)\Vert_{L^{b_1}(\R^3)}+ \Vert A_2(t)\Vert_{L^{b_2}(\R^3)}\big)\,.
 \end{split}
\end{equation}
Owing to the assumptions on $A$, $\Lambda\in L_{\mathrm{loc}}^{1}(\R,\ud t)$, therefore Gr\"onwall's lemma is applicable to \eqref{enel4} and we deduce
\[
 (\mathcal{E}(u_{\eps}))(t)\;\leqslant\;e^{\int_0^t\Lambda(s)\,\ud s}\Big((\mathcal{E}(u_{\eps}))(0)+\int_0^t\Lambda(s)\,\ud s\Big)\;\lesssim_{A,T}\;1\,,
\]
which proves \eqref{eq:Ebound}. Based on \eqref{pdelc} and on the norm equivalence \eqref{le:norm_equiv}, the bounds \eqref{eq:Mbound} and \eqref{eq:Ebound} then imply also \eqref{eq:uni_bou}.

Let us prove now the a priori bound \eqref{dissiesti}. Integrating \eqref{derieni} in $t\in[0,T)$  yields
\begin{equation*}
 \begin{split}
  (\mathcal{E}&(u_{\eps}))(T)-(\mathcal{E}(u_{\eps}))(0)\;= \\
  &=\;-\eps\!\int_0^T\!\!\Big(\int_{\R^3}\!\big(\,|(\nabla-\ii\,A)^2u_{\eps}|^2\,\ud x +\mathcal{R}(u_\varepsilon)(t)\Big)\,\ud t+\int_0^T\!\!\mathcal{S}(u_\varepsilon)(t)\,\ud t\,,
 \end{split}
\end{equation*}
whence 
\begin{equation*}
   \int_0^T\!\!\mathcal{R}(u_\varepsilon)(t)\,\ud t\;\leqslant\;\frac{1}{\varepsilon}\Big(\,|(\mathcal{E}(u_{\eps}))(T)-(\mathcal{E}(u_{\eps}))(0)|+\!\int_0^T\!\!|\mathcal{S}(u_\varepsilon)(t)|\,\ud t\Big)\,.
\end{equation*}
The bound \eqref{enel4} for $|\mathcal{S}(u_\varepsilon)(t)|$ and the bound \eqref{eq:Ebound} for $\mathcal{E}(u_\varepsilon)(t)$, together with the fact that $\Lambda\in L_{\mathrm{loc}}^{1}(\R,\ud t)$, then give
\begin{equation}\label{eq:boundintR}
  \int_0^T\!\!\mathcal{R}(u_\varepsilon)(t)\,\ud t\;\lesssim_{A,T}\;\varepsilon^{-1}\,.
\end{equation}
It is clear from \eqref{dopoderieni} that the l.h.s.~of the a priori bound \eqref{dissiesti} is controlled by  $\int_0^T\!\mathcal{R}(u_\varepsilon)(t)\ud t$, therefore \eqref{eq:boundintR} implies \eqref{dissiesti}.

This completes the proof under the additional assumption that $f\in\mathcal{S}(\R^3)$ and $A\in AC_{\mathrm{loc}}(\R,\mathcal{S}(\R^3))$. The proof in the general case of non-smooth potentials and non-smooth initial data follows by a density argument. We consider a sequence of regular potentials $A_n$ and regular initial data $f_n$ such that $f_n\rightarrow f$ in $H^1(\R^3)$ and $\|A_{n}-A\|_{\mathcal{A}_1}\to 0$ when $A\in\mathcal{A}_1$, or $\|A_{n}-A\|_{\mathcal{A}_2}\to 0$ when $A\in\mathcal{A}_2$, and we denote by $u_{\eps,n}$ the solution to the local Cauchy problem \eqref{eq:visc_CauMNLS} with initial datum $f_n$ and magnetic potential $A_n$.

Having already established Proposition \ref{pr:5ene} for such regular initial data and potentials, the bounds
\begin{eqnarray}
 \Vert u_{\eps,n}\Vert_{L^{\infty}([0,T);L^2(\R^3))}\!\!\!&\!\!\!\!\!\!\!\!\!\lesssim&\!\! 1 \label{eq:uni_bou_mass_n} \\
 \Vert u_{\eps,n}\Vert_{L^{\infty}([0,T);H^1(\R^3))}\!\!\!&\lesssim_{A,T}& \!\!1 \label{eq:uni_bou_n}
\end{eqnarray}
hold for every $n$ uniformly in $\eps>0$. The latter fact, together with the stability property
\begin{equation}\label{eq:stabinf2}
\|u_{n,\eps}-u_{\eps}\|_{L^{\infty}[0,T),H^1(\R^3))}\rightarrow 0\qquad\textrm{uniformly in $\varepsilon$} 
\end{equation}
given by Proposition \ref{prop:stab_nl}, then imply \eqref{eq:Mbound} and \eqref{eq:uni_bou} also in the general case. Analogously, since for fixed $t$ the mass $\mathcal{M}(u)(t)$ and the energy $\mathcal{E}(u)(t)$ depend continuously on the $H^1$-norm of $u(t)$, \eqref{eq:stabinf2} also implies \eqref{eq:Mbound} and \eqref{eq:Ebound} in the general case.

We are left to prove the energy a priori bound \eqref{dissiesti}. We first collect some useful facts, valid for a generic Strichartz pair $(q,r)$, with $r\in[2,3)$. The starting point is the stability result proved in Proposition \ref{prop:stab_nl}, which in this case reads
\begin{equation}\label{venti}
u_{n,\eps}\;\longrightarrow\; u_{\eps}\quad\mbox{in   }L^{q}([0,T),W^{1,r}(\mathbb{R}^3))\,.
\end{equation}
In particular,
\begin{eqnarray}
& u_{n,\eps}\;\longrightarrow\; u_{\eps}&\,\mbox{in }L^{q}([0,T),L^{\frac{Mr}{M-r}}(\mathbb{R}^3))\,,\quad M\in[3,+\infty]\,,\label{ventuno}\\
& \nabla u_{n,\eps}\;\longrightarrow\; \nabla u_{\eps}&\,\mbox{in }L^{q}([0,T),L^{r}(\mathbb{R}^3))\,.\label{ventidue}
\end{eqnarray}
Moreover the following identity is trivially satisfied (recall that $b_i>3)$:
\begin{equation}\label{puma}
(+\infty,b_i)*\Big(q,\frac{b_ir}{b_i-r}\Big)\;=\;(q,r),\quad i\in\{1,2\}\,.
\end{equation}
Now, \eqref{ventuno} and H\"older's inequality yield
\begin{equation}\label{ventitre}
 A u_{n,\eps}\longrightarrow A u_{\eps}\quad\mbox{in   }L^{q}([0,T),L^{r}(\mathbb{R}^3))\,,
 \end{equation}
and \eqref{ventidue} and \eqref {ventitre} yield
\begin{equation}\label{ventiquattro}
 |(\nabla-\ii A) u_{n,\eps}|^2\longrightarrow |(\nabla-\ii A) u_{\eps}|^{2}\quad\mbox{in   }L^{\frac{q}{2}}([0,T),L^{\frac{r}{2}}(\mathbb{R}^3)). 
 \end{equation}

We show now how to prove estimate \eqref{dissiesti} in the general case. Having already established Proposition \ref{pr:5ene} for regular initial data and potentials, we have in particular
\begin{eqnarray}
\big\|u_{n,\eps}^{\gamma-1}|(\nabla-\ii A) u_{n,\eps}|^2\big\|_{L^1([0,T),L^1(\mathbb{R}^3))}\!\!\!&\lesssim_{A,T}\;\eps^{-1},\label{StriStra}\\
\big\|(|x|^{-\alpha}*|u_{n,\eps}|^2)|(\nabla-\ii A) u_{n,\eps}|^2\big\|_{L^1([0,T),L^1(\mathbb{R}^3))}\!\!\!&\lesssim_{A,T}\;\eps^{-1},\label{StriStraconv}\\
\big\|(|x|^{-\alpha}*\nabla|u_{n,\eps}|^2)\nabla|u_{n,\eps}|^2\big\|_{L^1([0,T),L^1(\mathbb{R}^3))}\!\!\!&\lesssim_{A,T}\;\eps^{-1}\,.\label{StriStraconvdue}
\end{eqnarray}
For any $\gamma\in(1,5]$ we can find Strichartz pairs $(q_1,r_1)$ and $(q_2,r_2)$, with $r_1,r_2\in[2,3)$, such that
$$\textstyle\big(\frac{q_1}{\gamma-1},\frac{3r_1}{(3-r_1)(\gamma-1)}\big)*\big(\frac{q_2}{2},\frac{r_2}{2}\big)\;=\;(1,1)\,.$$
Then \eqref{ventuno}, \eqref{ventiquattro}, and H\"older's inequality yield
\begin{equation}\label{venticinque}
u_{n,\eps}^{\gamma-1}|(\nabla-\ii A) u_{n,\eps}|^2\!\!\longrightarrow u_{\eps}^{\gamma-1}|(\nabla-\ii A) u_{\eps}|^2\quad\mbox{in   }L^{1}([0,T),L^{1}(\mathbb{R}^3))\,, 
 \end{equation}
which together with the bound \eqref{StriStra} implies
\begin{equation}\label{StriStralimit}
\big\|u_{\eps}^{\gamma-1}|(\nabla-\ii A) u_{\eps}|^2\big\|_{L^1([0,T),L^1(\mathbb{R}^3))}\;\lesssim_{A,T}\;\eps^{-1}\,.
\end{equation}
In turn, the diamagnetic inequality $|\nabla|g||\leqslant |(\nabla-\ii A)g|$ and \eqref{StriStralimit} give also
\begin{equation}\label{StriStralimitdue}
\big\|u_{\eps}^{\gamma-1}|\nabla |u_{\eps}||^2\big\|_{L^1([0,T),L^1(\mathbb{R}^3))}\;\lesssim_{A,T}\;\eps^{-1}\,.
\end{equation}

Concerning the convolution terms, for any $\alpha\in(0,3)$ we can find Strichartz pairs $(\widetilde{q}_1,\widetilde{r}_1)$ and $(\widetilde{q}_2,\widetilde{r}_2)$, with $\widetilde{r}_1,\widetilde{r}_2\in[2,3)$, such that
\[
 \begin{split}
  \int_0^T\!\!\int_{\mathbb{R}^3}&(|x|^{-\alpha}*|u_{n,\eps}|^2)|(\nabla-\ii A) u_{n,\eps}|^2\,\ud x\,\ud t\;\lesssim \\
  &\;\lesssim\;\|u\|_{L^{\frac{2}{\widetilde{q}}}([0,T),L^{\frac{3\widetilde{r}_1}{2(3-\widetilde{r}_1)}}(\mathbb{R}^3))}\,\|(\nabla-\ii A) u_{n,\eps}\|_{L^{\frac{\widetilde{q}_2}{2}}([0,T),L^{\frac{\widetilde{r}_2}{2}}(\mathbb{R}^3))}\,,
 \end{split}
\]
which is obtained by the Hardy-Littlewood-Sobolev and H\"older's inequality. Therefore, \begin{equation}\label{ventisei}
\begin{split}
 (|x|^{-\alpha}*|u_{n,\eps}|^2)|&(\nabla-\ii A) u_{n,\eps}|^2\;\longrightarrow\; (|x|^{-\alpha}*|u_{\eps}|^2)|(\nabla-\ii A) u_{\eps}|^2 
\\
 &\qquad\qquad \mbox{in   }L^{1}([0,T)L^{1}(\mathbb{R}^3))\,, 
\end{split}
 \end{equation}
which together with the bound \eqref{StriStraconv} implies
\begin{equation}\label{StriStralimitconv}
\big\|(|x|^{-\alpha}*|u_{\eps}|^2)|(\nabla-\ii A) u_{\eps}|^2\big\|_{L^1([0,T),L^1(\mathbb{R}^3))}\;\lesssim_{A,T}\;\eps^{-1}\,.
\end{equation}
In analogous manner, using Hardy-Littlewood-Sobolev and H\"older's inequality, from \eqref{ventuno} and \eqref{ventidue} we get
\begin{equation}\label{ventisette}
\begin{split}
(|x|^{-\alpha}*\nabla|u_{n,\eps}|^2)&\nabla|u_{n,\eps}|^2\;\longrightarrow\; (|x|^{-\alpha}*\nabla|u_{\eps}|^2)\nabla|u_{\eps}|^2
\\
 &\quad\quad \mbox{in   }L^{1}([0,T)L^{1}(\mathbb{R}^3))\,, 
\end{split}
 \end{equation}
which together with the bound \eqref{StriStraconvdue} implies
\begin{equation}\label{StriStraconvduelimit}
\big\|(|x|^{-\alpha}*\nabla|u_{\eps}|^2)\nabla|u_{\eps}|^2\big\|_{L^1([0,T),L^1(\mathbb{R}^3))}\;\lesssim_{A,T}\;\eps^{-1}.
\end{equation}
The a priori abound \eqref{dissiesti} in the general case follows by combining \eqref{StriStralimit}, \eqref{StriStralimitdue}, \eqref{StriStralimitconv} and \eqref{StriStraconvduelimit}.
\end{proof}

\begin{remark}
 The inequality \eqref{pdelc}, namely
 \begin{equation}
 \Vert u_{\eps}(t)\Vert_{H_{A(t)}^1}^2\leqslant\; (\mathcal{M}(u_{\eps}))(t)+(\mathcal{E}(u_{\eps}))(t)\,,\qquad t\in[0,T)\,,
\end{equation}
 reflects the \emph{defocusing} structure of the regularised magnetic NLS \eqref{eq:visc_nls}.
\end{remark}

\section{Global existence for the regularised equation}\label{eq:eps-GWP}

In this Section we exploit the a priori estimates for mass and energy so as to prove that the local solution to the regularised Cauchy problem \eqref{eq:visc_CauMNLS}, constructed is Section \ref{sec:LWP}, can be actually extended globally in time.

We discuss first the result in the energy sub-critical case.

\begin{theorem}[Global well-posedness, energy sub-critical case]\label{t:s}
Assume that $A\in\mathcal{A}_1$ or $A\in\mathcal{A}_2$, and that the exponents in the non-linearity \eqref{hart_pp_nl} are in the regime $\gamma\in(1,5)$ and $\alpha\in(0,3)$. Let $\eps>0$. Then the regularised non-linear magnetic Schr\"{o}dinger equation \eqref{eq:visc_nls} is globally well-posed in $H^1(\R^3)$, in the sense of Definitions \ref{de:ws_solution_alt} and \ref{de:WP}. Moreover, the solution $u_{\eps}$ to \eqref{eq:visc_nls} with given initial datum $f\in H^1(\R^3)$ satisfies the bound 
\begin{equation}\label{eq:unif_bound}
\Vert u_{\eps}\Vert_{L^{\infty}[0,T],H^1(\R^3)}\;\lesssim_T\; 1\qquad\forall\, T\in(0,+\infty)\,,
\end{equation}
uniformly in $\eps>0$.
\end{theorem}

\begin{proof}
The local well-posedness is proved in Proposition \ref{pr:exsub}. Because of \eqref{eq:uni_bou}, the $H^1$-norm of $u_{\eps}$ is bounded on finite intervals of time. Therefore, by the blow-up alternative, the solution is necessarily global and in particular it satisfies the bound \eqref{eq:unif_bound}.
\end{proof}

We discuss now the analogous result in the energy-critical case.

\begin{theorem}[Global existence and uniqueness, energy critical case]\label{t:c}
Assume that $A\in\mathcal{A}_1$ or $A\in\mathcal{A}_2$, and that the exponents in the non-linearity \eqref{hart_pp_nl} are in the regime $\gamma=5$ and $\alpha\in(0,3)$. Let $\eps>0$ and $f\in H^1(\R^3)$. The Cauchy problem \eqref{eq:visc_CauMNLS} has a unique global strong $H^1$-solution $u_\varepsilon$, in the sense of Definition \ref{de:ws_solution_alt}. Moreover, $u$ satisfies the bound 
\begin{equation}\label{eq:unif_bound_cri}
\Vert u_{\eps}\Vert_{L^{\infty}[0,T],H^1(\R^3)}\;\lesssim_T\; 1\qquad\forall\, T\in(0,+\infty)\,,
\end{equation}
uniformly in $\eps>0$.
\end{theorem}

\begin{proof}
The existence of a unique local solution $u_\varepsilon$ is proved in Proposition \ref{pr:crie}. The a priori bound \eqref{dissiesti} implies that
$$\int_0^T\!\!\int_{\mathbb{R}^3}\big(\,|u_{\eps}|^2\,\nabla|u_{\eps}|\big)^2\,\ud x\,\ud t\;\lesssim\;\varepsilon^{-1}\,,$$
which, together with Sobolev's embedding, yields 
\begin{equation}\label{eq:seidiciotto}
\begin{aligned}
\Vert u_{\eps}\Vert^6_{L^6([0,T],L^{18}(\R^3))}\;&=\;\Vert u_{\eps}^3\Vert_{L^2([0,T],L^{6}(\R^3))}^2\;\lesssim\int_0^T\!\!\int_{\R^3}|\nabla|u_{\eps}|^3|^2\,\ud x\,\ud t \\
&\lesssim\int_0^T\!\!\int_{\R^3}|u_{\eps}|^4\,|\nabla|u_{\eps}||^2\,\ud x\,\ud t\;\lesssim\;\varepsilon^{-1}<\;+\infty\,.
\end{aligned}
\end{equation}
Owing to \eqref{eq:seidiciotto} and to the blow-up alternative proved in Proposition \ref{pr:crie}, we conclude that the solution $u$ can be extended globally and moreover, using again \eqref{eq:uni_bou}, it satisfies the bound \eqref{eq:unif_bound_cri}.
\end{proof}

\begin{remark}\label{asba}
As anticipated in the Introduction, right after stating the assumptions on the magnetic potential, let us comment here about the fact that in the \emph{mass sub-critical} regime ($\gamma\in(1,\frac{7}{3})$ and $\alpha\in(0,2)$) we can work with the larger class $\widetilde{\mathcal{A}}_1$ instead of $\mathcal{A}_1$ and still prove the extension of the local solution globally in time with finite $H^1$-norm on arbitrary finite time interval. This is due to the fact that, for a potential $u\in\widetilde{\mathcal{A}}_1$ and in the mass sub-critical regime, in order to extend the solution globally neither need we the estimate \eqref{eq:uni_bou} as in the proof of Theorem \ref{t:s}, nor need we the estimate \eqref{dissiesti} as in the proof of Theorem \ref{t:c}. Indeed, we can first prove local well-posedness in $L^2(\mathbb{R}^3)$ for the regularised magnetic NLS \eqref{eq:visc_nls}, using a fixed point argument based on the space-time estimates for the heat-Schr\"odinger flow, in the very same spirit of the proof of Theorem \ref{th:energy_linear}. Then we can extend such a solution globally in time using only the mass a priori bound \eqref{eq:Mbound}, for proving such a bound does not require any time-regularity assumption on the magnetic potential. Moreover, since the non-linearities are mass sub-critical and since we can prove convenient estimates on the commutator $[\nabla,(\nabla-\ii\,A)^2]$ when $\max{\{b_1,b_2\}}\in (3,6)$, we can show that the global $L^2$-solution exhibits \emph{persistence of $H^1$-regularity} in the sense that it stays in $H^1(\R^3)$ for every positive time provided that the initial datum belongs already to $H^1(\R^3)$. This way, we obtain existence and uniqueness of one global strong $H^1$-solution.
\end{remark}

%
%
%

\section{Removing the regularisation}\label{sec:remov-reg}

In this Section we prove our main Theorem \ref{th:main}. The proof is based on a compactness argument that we develop in Subsection \ref{sec:LocalWeakSol}, so as to remove the $\varepsilon$-regularisation, and leads to a \emph{local} weak $H^1$-solution to \eqref{eq:CauMNLS}.

The reason why by compactness we can only produce local solutions is merely due to the local-in-time regularity of magnetic potentials belonging to the class $\mathcal{A}_1$ or $\mathcal{A}_2$ -- \emph{globally}-in-time regular potentials, say, $AC(\mathbb{R})$-potentials, would instead allow for a direct removal of the regularisation globally in time. 

In order to circumvent this simple obstruction, in Subsection \ref{sec:proof_of_main_Thm} we work out a straightforward `gluing' argument, eventually proving Theorem \ref{th:main}.

\subsection{Local weak solutions}\label{sec:LocalWeakSol}

The main result of this Subsection is the following.

\begin{proposition}\label{pr:compa}
Assume that $A\in\mathcal{A}_1$ or $A\in\mathcal{A}_2$, and that the exponents in the non-linearity \eqref{hart_pp_nl} are in the whole regime $\gamma\in(1,5]$ and $\alpha\in(0,3)$. Let $T>0$, and $f\in H^1(\R^3)$. For any sequence $(\varepsilon_n)_n$ of positive numbers with  $\eps_n\downarrow 0$, let $u_n$ be the unique global strong $H^1$-solution to the Cauchy problem \eqref{eq:visc_CauMNLS} with viscosity parameter $\eps=\eps_n$ and with initial datum $f$, as provided by Theorem \ref{t:s} in the energy sub-critical case and by Theorem \ref{t:c} in the energy critical case. Then, up to a subsequence, $u_n$ converges weakly-$*$ in $L^{\infty}([0,T],H^1(\R^3))$ to a local weak $H^1$-solution $u$ to the magnetic NLS \eqref{eq:magneticNLS} in the time interval $[0,T]$ and with initial datum $f$.
\end{proposition}

In order to set up the compactness argument that proves Proposition \ref{pr:compa} we need a few auxiliary results, as follows.

\begin{lemma}\label{le:aabb}
The sequence $(u_n)_n$ in the assumption of Proposition \ref{pr:compa} is bounded in $L^{\infty}([0,T],H^1(\R^3))$, i.e.,
\begin{equation}\label{uni_bou_n}
\Vert u_n\Vert_{L^{\infty}([0,T],H^1(\R^3))}\;\lesssim_{A,T}\; 1\,,
\end{equation}
and hence, up to a subsequence, $(u_n)_n$  admits a weak-$*$ limit $u$ in $L^{\infty}([0,T],H^1(\R^3))$.
\end{lemma}

\begin{proof}
An immediate consequence of the uniform-in-$\varepsilon$ bounds \eqref{eq:unif_bound}-\eqref{eq:unif_bound_cri} and the the Banach-Alaoglu Theorem.
\end{proof}

\begin{lemma}\label{lann}
For the sequence $(u_n)_n$ in the assumption of Proposition \ref{pr:compa}
there exist indices $p_i,p_{ij}\in[\frac{6}{5},2]$, $i,j\in\{1,2\}$, such that
\begin{equation}\label{lima_uno}
(A_i\cdot\nabla u_n)_n\mbox{ is a bounded sequence in }L^{\infty}([0,T],L^{p_i}(\R^3))\,,
\end{equation}
\begin{equation}\label{lima_due}
(A_i\cdot A_ju_n)_n\mbox{ is a bounded sequence in }L^{\infty}([0,T],L^{p_{ij}}(\R^3))\,.
\end{equation}
\end{lemma}

\begin{proof}
For $p_i:=b_i*2\in[\frac{6}{5},2]$, $i\in\{1,2\}$, the bound \eqref{uni_bou_n} and H\"older's inequality give
$$\Vert A_i\cdot\nabla u_n\Vert_{L^{\infty}([0,T],L^{p_i}(\R^3))}\;\lesssim\;\Vert A_i\Vert_{L^{\infty}([0,T],L^{b_i}(\R^3))}\Vert \nabla u_n\Vert_{L^{\infty}([0,T],L^2(\R^3))}\;\lesssim_{A,T}\; 1\,,$$
which proves \eqref{lima_uno}. Moreover, there exist $M_{ij}\in[2,6]$, $i,j\in\{1,2\}$, such that $p_{ij}:=b_i*b_j*M_{ij}\in[\frac{6}{5},2]$, therefore the bound \eqref{uni_bou_n}, H\"older's inequality, and Sobolev's embedding give
\begin{gather*}
\Vert A_i\cdot A_j u_n\Vert_{L^{\infty}([0,T],L^{p_{ij}}(\R^3))}\lesssim\\
\Vert A_i\Vert_{L^{\infty}([0,T],L^{b_i}(\R^3))}\Vert A_j\Vert_{L^{\infty}([0,T],L^{b_j}(\R^3))}\Vert  u_n\Vert_{L^{\infty}([0,T],L^{M_{ij}}(\R^3))}\lesssim_{A,T} 1,
\end{gather*}
which proves \eqref{lima_due}.
\end{proof}


\begin{lemma}\label{linn}
For the sequence $(u_n)_n$ in the assumption of Proposition \ref{pr:compa}, and for every $\gamma\in(1,5]$ and $\alpha\in(1,3)$, there exist indices $p(\gamma),\widetilde{p}(\alpha)\in[\frac{6}{5},2]$ such that
\begin{equation}\label{lima_tre}
 \big(|u_n|^{\gamma-1}u_n\big)_n\mbox{ is a bounded sequence in }L^{\infty}([0,T],L^{p(\gamma)}(\R^3))\,,
\end{equation}
\begin{equation}\label{lima_quattro}
 \big((\,|\cdot|^{-\alpha}*|u_n|^2)u_n\big)_n\mbox{ is a bounded sequence in }L^{\infty}([0,T],L^{\widetilde{p}(\alpha)}(\R^3))\,.
\end{equation}

\end{lemma}

\begin{proof}
For any $\gamma\in(1,5]$ there exists $M:=M(\gamma)\in[2,6]$ such that $M/\gamma\in[\frac{6}{5},2]$, whence
\[
 \begin{split}
  \Vert |u_n|^{\gamma-1}u\Vert_{L^{\infty}([0,T],L^{M/\gamma}(\mathbb{R}^3))}\;&\leqslant\; \Vert u_n\Vert_{L^{\infty}([0,T],L^M(\mathbb{R}^3))}^{\gamma} \\
  &\lesssim\;\Vert u_n\Vert_{L^{\infty}([0,T],H^1(\mathbb{R}^3))}^{\gamma}\;\lesssim_{A,T}\;1\,,
 \end{split}
\]
based on the bound \eqref{uni_bou_n} and Sobolev's embedding, which proves \eqref{lima_tre}, with $p(\gamma):=M/\gamma$.
Next, let us use the Hardy-Littlewood-Sobolev inequality, for $m(\alpha)\in(1,\frac{3}{3-\alpha})$ and $g\in L^{m(\alpha)}(\mathbb{R}^3)$,
$$\big\Vert\,|\cdot|^{-\alpha}*g\,\big\Vert_{L^{q(m(\alpha))}(\mathbb{R}^3)}\;\lesssim\; \Vert g\Vert_{L^{m(\alpha)}(\mathbb{R}^3)}\,,\qquad q(m):=\textstyle\frac{3m(\alpha)}{3-(3-\alpha)m(\alpha)}\,.$$
Taking
\begin{equation}\label{qm}
 \begin{array}{ll}
  m(\alpha)\in(1,{\textstyle\frac{3}{3-\alpha}})\;&\quad\textrm{if }\alpha\in(0,2] \\
  m(\alpha)\in(1,3]\;&\quad\textrm{if }\alpha\in(2,3)\,,
 \end{array}
\end{equation}
the Hardy-Littlewood-Sobolev inequality above and Sobolev's embedding yield
\begin{equation}\label{mangf}
\begin{split}
 \big\Vert\,|\cdot|^{-\alpha}*|u|^2\big\Vert_{L^{\infty}([0,T],L^{q(m(\alpha))}(\mathbb{R}^3))}\;&\lesssim\; \Vert u^2
\Vert_{L^{\infty}([0,T],L^{m(\alpha)}(\mathbb{R}^3))} \\
&\lesssim\; \Vert u\Vert_{L^{\infty}([0,T],H^1(\mathbb{R}^3))}^2\,.
\end{split}
\end{equation}
Since $\frac{3}{4-\alpha}<1$ for $\alpha\in(0,3)$, we can find $m(\alpha)$ that satisfies \eqref{qm} as well as $q(m(\alpha))*2\in[\frac{6}{5},2]$, namely
\begin{equation}
m(\alpha)\;\in\;\Big(\frac{3}{4-\alpha},\frac{3}{3-\alpha}\Big)\,.
\end{equation}
As a consequence, for $\widetilde{p}(\alpha):=q(m(\alpha))*2\in[\frac{6}{5},2]$ one has 
\[
 \begin{split}
  \big\Vert(\,|\cdot|^{-\alpha}&*|u_n|^2)u_n\big\Vert_{L^{\infty}([0,T],L^{\widetilde{p}(\alpha)}(\mathbb{R}^3))}\;\lesssim \\
  &\lesssim\;\big\Vert\,|\cdot|^{-\alpha}*|u_n|^2\big\Vert_{L^{\infty}([0,T],L^{q(m(\alpha))}(\mathbb{R}^3))}\,\Vert u_n\Vert_{L^{\infty}([0,T],L^{2}(\mathbb{R}^3))} \\
  &\lesssim_{A,T}\;\Vert u_n\Vert_{L^{\infty}([0,T],H^{1}(\mathbb{R}^3))}\;\lesssim_{A,T}\; 1\,,
 \end{split}
\]
based on H\"{o}lder's inequality (first step), the bound \eqref{mangf} (second step), and Sobolev's embedding (third step),  which proves \eqref{lima_quattro}.
\end{proof}

\begin{corollary}\label{cor:cor}
For the sequence $(u_n)_n$ in the assumption of Proposition \ref{pr:compa} there exist indices $p_i$, $p_{ij}$, $p(\gamma)$, and $\widetilde{p}(\alpha)$ in $[\frac{6}{5},2]$, and there exists functions 
$X_i\in L^{\infty}([0,T],L^{p_i}(\R^3)) $, $Y_{ij}\in L^{\infty}([0,T],L^{p_{ij}}(\R^3))$, $N_1\in L^{\infty}([0,T],L^{p(\gamma)}(\R^3))$, and $N_2\in L^{\infty}([0,T],L^{\widetilde{p}(\alpha)}(\R^3))$ such that
\begin{eqnarray}
 A_i\cdot\nabla u_n\!\!\!&\rightarrow&\!\!\! X_i\qquad\mbox{weakly-$*$}\mbox{ in } L^{\infty}([0,T],L^{p_i}(\R^3)) \label{eq:cADu}\\
 A_i\cdot A_ju_n\!\!\!&\rightarrow&\!\!\! Y_{ij}\qquad\mbox{weakly-$*$}\mbox{ in } L^{\infty}([0,T],L^{p_{ij}}(\R^3)) \label{eq:cAAu} \\
 |u_n|^{\gamma-1}u_n\!\!\!&\rightarrow&\!\!\! N_1\qquad\mbox{weakly-$*$}\mbox{ in } L^{\infty}([0,T],L^{p(\gamma)}(\R^3)) \label{eq:cuuu}\\
 (|\cdot|^{-\alpha}*|u_n|^2)u_n\!\!\!&\rightarrow&\!\!\! N_2\qquad\mbox{weakly-$*$}\mbox{ in }L^{\infty}([0,T],L^{\widetilde{p}(\alpha)}(\R^3))\,. \label{eq:calphau}
\end{eqnarray}
\end{corollary}

\begin{proof}
 An immediate consequence of Lemmas \ref{lann} and \ref{linn}, using the Banach-Alaoglu Theorem.
\end{proof}

\begin{lemma}
For the sequence $(u_n)_n$ in the assumption of Proposition \ref{pr:compa}, for the corresponding weak limit $u$ identified in Lemma  \ref{le:aabb}, and for the exponents $p_i$, $i\in\{1,2\}$ identified in Corollary \ref{cor:cor}, one has
\begin{equation}\label{eq:cADuF}
 A_i\cdot\nabla u_n\;\rightarrow\;A_i\cdot\nabla u\qquad\textrm{weakly in } L^{2}([0,T],L^{p_i}(\R^3))\,.
\end{equation}
\end{lemma}

\begin{proof}
Because of the bound \eqref{uni_bou_n}, up to a subsequence
$$\nabla u_n\rightarrow \nabla u\qquad\mbox{weakly in }L^2([0,T],L^2(\R^3))\,.$$
Now, since $p_i=b_i*2$ and hence $p_i'*b_i=2$, and since $A_i\in L^\infty([0,T]L^{b_i}(\R^3))$, one has $A_i\eta\in L^2([0,T],L^2(\R^3))$ for any  $\eta\in L^{2}([0,T],L^{p'_i}(\mathbb{R}^3))$. Then
$$\int_0^T\!\!\int_{\R^3}A_i\cdot(\nabla u_n-\nabla u)\overline{\eta}\,\ud x\,\ud t\;=\;\int_0^T\!\!\int_{\R^3}(\nabla u_n-\nabla u)A_i\overline{\eta}\,\ud x\,\ud t\;\rightarrow\; 0\,,$$
thus concluding the proof.
\end{proof}


\begin{lemma}
Let $\Omega$ be an open, bounded subset of $\R^3$ and let $M\in[1,+\infty]$. For the sequence $(u_n)_n$ in the assumption of Proposition \ref{pr:compa}, and for the corresponding weak limit $u$ identified in Lemma  \ref{le:aabb},
\begin{equation}\label{eq:unuO}
 u_n|_\Omega\rightarrow u|_\Omega\qquad\mbox{strongly in }L^M([0,T],L^4(\Omega))\,.
\end{equation}
\end{lemma}

\begin{proof}
Because of the \eqref{uni_bou_n}, $(u_n|)_n$ is a bounded sequence in $L^M([0,T],H^1(\R^3))$ for any $M\in[1,+\infty]$.
Moreover, for every time $t\in[0,T]$ $u_n$ satisfies
\[
 \ii\,\partial_t u_n\;=\;-(1-\ii\,\varepsilon)(\Delta u_n-2\,\ii\,A\cdot\nabla u_n-|A|^2u_n)+\mathcal{N}(u_n)
\]
as an identity between $H^{-1}$ functions.
Hence, owing to the estimate \eqref{H1H-1} and to the boundedness of the map $\mathcal{N}(u):H^1(\mathbb{R}^3)\to H^{-1}(\mathbb{R}^3)$,
\begin{equation}\label{uni:dersuR3}
 \begin{aligned}
 \Vert&\partial_tu_n\Vert_{L^\infty([0,T],H^{-1}(\mathbb{R}^3))}\; \lesssim\;\\
 & \lesssim\;\Vert(\nabla-\ii\,A)^2u_n\Vert_{L^{\infty}([0,T],H^{-1}(\mathbb{R}^3))}+\Vert \mathcal{N}(u_n)\Vert_{L^{\infty}([0,T],H^{-1}(\mathbb{R}^3))}\\
&\lesssim_A\,\Vert u_n\Vert_{L^{\infty}([0,T],H^1(\mathbb{R}^3))}\;\lesssim_{A,T}\; 1\,.
\end{aligned}
\end{equation}
In particular,
\begin{equation}\label{uni:der}
 \begin{aligned}
 \Vert&\partial_tu_n|_\Omega\Vert_{L^1([0,T],H^{-1}(\Omega))}\; \lesssim_{A,T}\; 1\,.
\end{aligned}
\end{equation}
Therefore \eqref{eq:unuO} follows as an application of Aubin-Lions compactness lemma (see, e.g., \cite[Section 7.3]{Roubicek_NPDEbook}) to the bound \eqref{uni:der} and with respect to the compact inclusion $H^1(\Omega)\hookrightarrow L^4(\Omega)$ and the continuous inclusion 
$L^4(\Omega)\hookrightarrow H^{-1}(\Omega)$.
\end{proof}

\begin{lemma}\label{pr:alme}
For the limit function $u$ identified in Lemma  \ref{le:aabb} and for the limit functions $X_i$, $Y_{ij}$, and $N_i$ identified in Corollary \ref{cor:cor} one has the pointwise identities for $t\in[0,T]$ and a.e.~$x\in\mathbb{R}^3$:
\begin{eqnarray}
 A_i\cdot\nabla u\!\!\!&=&\!\!\!X_i \label{mufy1} \\
 A_i\cdot A_ju\!\!\!&=&\!\!\!Y_{ij} \label{mufy2} \\
 |u|^{\gamma -1}u\!\!\!&=&\!\!\!N_1  \label{mufy3} \\
  \big(\,|\cdot|^{-\alpha}*|u|^2\big)u\!\!\!&=&\!\!\!N_2\,. \label{mufy4}
\end{eqnarray}
\end{lemma}

\begin{proof}
For the sequence $(u_n)_n$ in the assumption of Proposition \ref{pr:compa}, and for the exponents $p_i$, $i\in\{1,2\}$ identified in Corollary \ref{cor:cor}, one has
\begin{equation}
 A_i\cdot\nabla u_n\;\rightarrow\;A_i\cdot\nabla u\qquad\textrm{weakly in } L^{2}([0,T],L^{p_i}(\R^3))\,.
\end{equation}
Indeed, because of the bound \eqref{uni_bou_n}, up to a subsequence
$$\nabla u_n\rightarrow \nabla u\qquad\mbox{weakly in }L^2([0,T],L^2(\R^3))\,;$$
therefore, since $p_i=b_i*2$ and hence $p_i'*b_i=2$, and since $A_i\in L^\infty([0,T]L^{b_i}(\R^3))$, one has $A_i\eta\in L^2([0,T],L^2(\R^3))$ for any  $\eta\in L^{2}([0,T],L^{p'_i}(\mathbb{R}^3))$,
$$\int_0^T\!\!\int_{\R^3}A_i\cdot(\nabla u_n-\nabla u)\overline{\eta}\,\ud x\,\ud t\;=\;\int_0^T\!\!\int_{\R^3}(\nabla u_n-\nabla u)A_i\overline{\eta}\,\ud x\,\ud t\;\rightarrow\; 0\,.$$
The limits \eqref{eq:cADu} and \eqref{eq:cADuF} imply
\[
 \begin{split}
  \int_0^T\!\!\int_{\R^3}\big(A_i\cdot\nabla u_n-A_i\cdot\nabla u \big)\,\varphi\,\ud x\,\ud t\;&\rightarrow\;0\\
    \int_0^T\!\!\int_{\R^3}\big(A_i\cdot\nabla u_n-X_i \big)\,\varphi\,\ud x\,\ud t\;&\rightarrow\;0
 \end{split}
\]
for arbitrary $\varphi\in\mathcal{S}(\R\times\R^3)$,
whence the pointwise identity \eqref{mufy1}.
Let now $\Omega$ be an open and bounded subset of $\mathbb{R}^3$, and let $M\in[1,+\infty]$. Since, as seen in \eqref{eq:unuO}, $u_n|_\Omega$ converges to $u|_\Omega$ in $L^M([0,T],L^4(\Omega))$, then up to a subsequence one has also pointwise convergence, whence
\begin{eqnarray}
 A_i\cdot A_ju_n|_\Omega\!\!\!&\rightarrow&\!\!\! A_i\cdot A_ju|_\Omega \label{eq:cAAuF} \\
 |u_n|^{\gamma -1}u_n|_\Omega\!\!\!&\rightarrow&\!\!\! |u|^{\gamma -1}u|_\Omega \label{eq:cuuuF} \\
 \big(\,|\cdot|^{-\alpha}*|u_n|^2\big)u_n|_\Omega\!\!\!&\rightarrow&\!\!\! \big(\,|\cdot|^{-\alpha}*|u|^2\big)u|_\Omega \label{eq:calphauF}
\end{eqnarray}
pointwise for $t\in[0,T]$ and a.e.~$x\in\Omega$.
Therefore,
\eqref{mufy2}, \eqref{mufy3}, and \eqref{mufy4}
follow by the uniqueness of the pointwise limit and the arbitrariness of $\Omega$, combining, respectively, 
\eqref{eq:cAAu}, \eqref{eq:cuuu}, and \eqref{eq:calphau}
with, respectively,
\eqref{eq:cAAuF}, \eqref{eq:cuuuF}, and \eqref{eq:calphauF}.
\end{proof}

With the material collected so far we can complete the argument for the removal of the parabolic regularisation, locally in time.

\begin{proof}[Proof of Proposition \ref{pr:compa}]
We want to show that the function $u$ identified in Lemma \ref{le:aabb} is actually a local weak $H^1$-solution, in the sense of Definition \ref{de:ws_solution_alt} to the magnetic NLS \eqref{eq:magneticNLS} with initial datum $f$ in the time interval $[0,T]$. 
All the exponents  $p_i$, $p_{ij}$, $p(\gamma)$ and $\widetilde{p}(\alpha)$ identified in Corollary \ref{cor:cor} belong to the interval $[\frac{6}{5},2]$, and then by Sobolev's embedding the functions $X_i=A_i\cdot\nabla u$, $Y_{ij}=A_i\cdot A_ju$, $N_1=|u|^{\gamma -1}u$, and $N_2=(|\cdot|^{-\alpha}*u^2)u$ discussed in Corollary \ref{cor:cor} and Lemma \ref{pr:alme} all belong to $H^{-1}(\R^3)$, and so too does $\Delta u$, obviously. Therefore \eqref{eq:magneticNLS} is satisfied by $u$ as an identity between $H^{-1}$-functions. As a consequence, one can repeat the argument used to derive the estimate \eqref{uni:dersuR3}, whence $\partial_t u\in L^{\infty}([0,T],H^{-1}(\mathbb{R}^3))$. Thus, $u\in W^{1,\infty}([0,T],H^{-1}(\mathbb{R}^3))$. On the other hand $u_n\in C^1([0,T],H^{-1}(\mathbb{R}^3))$, and Lemma \ref{le:aabb} implies
\[
 \int_0^T\!\!\int_{\mathbb{R}^3}\eta(t,x)\big(u_n(t,x)-u(t,x) \big)\,\ud x\,\ud t\;\to\;0\qquad\forall \eta\in L^1([0,T],H^{-1}(\mathbb{R}^3)\,.
\]
For $\eta(t,x)=\delta(t-t_0,x)\varphi(x)$, where $t_0$ is arbitrary in $[0,T]$ and  $\varphi$ is arbitrary in $L^2(\mathbb{R}^3)$, the limit above reads $u_n(t_0,\cdot)\to u(t_0,\cdot)$ weakly in $L^2(\mathbb{R}^3)$, whence $u(0,\cdot)=f(\cdot)$.
\end{proof}

\subsection{Proof of the main Theorem}\label{sec:proof_of_main_Thm}

It is already evident at this stage that had we assumed the magnetic potential to be an $AC$-function for all times, then the proof of the existence of a global weak solution with finite energy would be completed with the proof of Proposition \ref{pr:compa} above, in full analogy with the scheme of the work \cite{GuoNakStr-1995} we mentioned in the Introduction.

Our potential being in general only $AC_{\mathrm{loc}}$ in time, we cannot appeal to bounds that are uniform in time (indeed, our \eqref{eq:unif_bound} and \eqref{eq:unif_bound_cri} are $T$-dependent), and the following straightforward strategy must be added in order to complete the proof of our main result.

\begin{proof}[Proof of Theorem \ref{th:main}]
We set $T=1$ and we choose an arbitrary sequence $(\eps_n)_n$ of positive numbers with $\eps_n\downarrow 0$. Let $u_n$ be the unique local strong $H^1$-solution to the regularised magnetic NLS \eqref{eq:visc_nls} with viscosity parameter $\eps=\eps_n$ and with initial datum $f\in H^1(\R^3)$. By Proposition \ref{pr:compa}, there exists a subsequence $(\eps_{n'})_{n'}$ of $(\eps_n)_n$ such that $u_{n'}\to u_1$ weakly-$*$ in $L^{\infty}([0,1],H^1(\R^3))$, where $u_1$ is a local weak $H^1$-solution to the magnetic NLS \eqref{eq:magneticNLS} with $u_1(0)=f$. If we take instead $T=2$ and repeat the argument, we find a subsequence $(\eps_{n''})_{n''}$ of $(\eps_{n'})_{n'}$ such that $u_{n''}\to u_2$ weakly-$*$ in $L^{\infty}([0,2],H^1(\R^3))$, where $u_2$ is a local weak $H^1$-solution to \eqref{eq:magneticNLS} with $u_2(0)=f$, now in the time interval $[0,2]$. Moreover, having refined the $u_{n'}$'s in order to obtain the $u_{n''}$'s, necessarily $u_2(t)=u_1(t)$ for $t\in[0,1]$. Iterating this process, we construct for any $N\in\N$ a function $u_N$ which is a local weak $H^1$-solution to \eqref{eq:magneticNLS} in the time interval $[0,N]$, with $u_N(0)=f$ and $u_N(t)=u_{N-1}(t)$ for $t\in[0,N-1]$. It remains to define 
$$u(t,x)\;:=\;u_{N}(t,x)\qquad x\in\mathbb{R}^3\,,\quad t\in[0,+\infty)\,\quad N=[t]\,.$$
Since $u_{N}\in L^{\infty}([0,N],H^1(\R^3))\cap W^{1,\infty}([0,N],H^{-1}(\R^3))$ for every $N\in N$, such $u$ turns out to be a global weak $H^1$-solution to \eqref{eq:CauMNLS} with finite energy for a.e. $t\in \R$, uniformly on compact time intervals.
\end{proof}

%

\def\cprime{$'$}

\end{document}